\documentclass[12pt, oneside]{article}
\usepackage{graphicx}
\usepackage{curves}
\usepackage{amsmath}
\usepackage{amssymb}
\usepackage{amsfonts}
\usepackage{amsthm}
\usepackage{longtable}
\usepackage{multirow}
\usepackage{color}

\newtheorem{lemma}{Lemma}[section]
\newtheorem{theorem}[lemma]{Theorem}
\newtheorem{prop}[lemma]{Proposition}

\newtheorem{coro}[lemma]{Corollary}

\newtheorem{definition}[lemma]{Definition}

\newtheorem{problem}{Problem}

\newcommand*\xbar[1]{%
  \hbox{%
    \vbox{%
      \hrule height 0.5pt 
      \kern0.5ex
      \hbox{%
        \kern-0.1em
        \ensuremath{#1}%
        \kern-0.1em
      }%
    }%
  }%
}

\newcommand{\po}{\mathcal{P}}

\newcommand{\G}{\mathcal{G}}

\newcommand{\fl}{\mathcal{F}}

\newcommand{\M}{\mathcal{M}}
\newcommand{\m}{\mathcal{M}}
\newcommand{\n}{[n]}

\begin{document}

\title{Polytopality of Maniplexes}

\author{
Jorge Garza-Vargas and Isabel Hubard\footnote{isahubard@im.unam.mx} \\
Instituto de Matem\'aticas\\ Universidad Nacional Aut\'onoma de M\'exico
}

\maketitle

\begin{abstract}
Given an abstract polytope $\po$, its flag graph is the edge-coloured graph whose vertices are the flags of $\po$ and the $i$-edges correspond to $i$-adjacent flags. Flag graphs of polytopes are maniplexes. On the other hand, given a maniplex $\m$, on can define a poset $\po_\m$ by means of the non empty intersection of its faces. In this paper we give necessary and sufficient conditions (in terms of graphs) on a maniplex $\m$ in order for $\po_\m$ to be an abstract polytope. Moreover, in such case, we show that $\m$ is isomorphic to the flag graph of $\po_\m$. This in turn gives necessary and sufficient conditions for a maniplex to be (isomorphic to) the flag graph of a polytope.
\end{abstract}


\section{Introduction}

The beauty and symmetry of the Platonic Solids have been studied for centuries. 
These solids have been generalized, for example, to higher dimensions as convex polytopes (\cite{shla}), and to polyhedra surfaces different than the sphere as maps (see for example \cite{brah}).
Abstract polytopes generalize convex polytopes (or more precisely, their face lattice) to more general combinatorial objects and are defined as posets with certain conditions. 

Maniplexes were first introduced by Steve Wilson in \cite{mani} to somehow unify the study of maps and abstract polytopes. 
They generalize maps on surfaces to higher dimensions and, at the same time, they generalize (the flag graphs of) abstract polytopes. Of course, the Platonic Solids can be thought as maniplexes.

The task of determining if a maniplex is the flag graph of a polytope or not, can arise when dealing with operations on polytopes, such as the Petrie operation (\cite{sdmono}), the mix or parallel product (\cite{mixing}) or the Twist operation (\cite{twist}). The problem is also often encounter when dealing with covers and quotients of polytopes (\cite{semisparse,norm,regcov}), in particular in determining the polytopality of the so-called minimal regular cover of a polytope (\cite{mixing,tomo}).

Maniplexes can be defined in several ways. In this paper, they are regarded as edge-coloured graphs with certain properties. 
Given a maniplex, one can define, in a natural way, a poset associated to it. Such poset, in general, need not be a polytope.
The aim of the paper is to give necessary and sufficient conditions on an edge-coloured graph to be the flag graph of an abstract polytope. 
In fact, our results are slightly stronger as we give necessary and sufficient conditions on a maniplex in order for the induced poset to be a polytope.
We give two such conditions, that we call the {\em connected intersection property} and the {\em path intersection property}, respectively. As their names indicate, the first condition deals with certain connected components of the graph, while the second one deals with certain paths of it.

The paper is organized as follows. In Section~\ref{basic} we give the basic notions of abstract polytopes, maps, edge-coloured graphs and maniplexes, and set the notation used throughout the paper. Section~\ref{posets} deals with defining a poset induced by a given maniplex, as well as
 with the possible complications for a maniplex to be the graph flag of a polytope; some examples are given. In Section~\ref{sec:CIP} we revisit a defining property of abstract polytopes to obtain the connected intersection property and show that such condition charaterizes not only flag graphs of polytopes, but also maniplexes whose induced order is a polytope. In Section~\ref{sec:PIP} we re-write the connected intersection property in terms of the paths of the graph to obtain the path intersection property. We finish the paper with a short section on the mix (or parallel product) of maniplexes.

Throughout the paper it shall prove convenient to denote the set $\{0,1, \dots, n-1\}$ simply by $[n]$, and if $A\subset\n$, then we use $\overline{A}$ to denote the set $\n\setminus A$.

\section{Basic notions}\label{basic}

In this section we introduce the reader to the basic notions of abstract polytopes, maps, edge-coloured graphs and maniplexes.

\subsection{Abstract polytopes}\label{sec:poly}

Convex polytopes generalize the notion of polyhedra for higher dimensions. 
Abstract polytopes are combinatorial objects whose incidence structure resemble the incidence structure of convex polytopes. 
In fact, each convex polytope can be regarded as an abstract polytope. 
We give the basics about abstract polytopes and refer the reader to \cite{ARP} for more details of their study.

An {\em (abstract) $n$-polytope} (or a polytope of rank $n$) is a ranked partially ordered set,  $(\mathcal{P}, \leq)$ whose elements are called {\em faces. }
$\mathcal{P}$ must have a unique maximal face which  has rank $n$ and a unique minimal face which has rank $-1$, all the other faces are distributed in the remaining $n$ levels that go from $0$ to $n-1$. 
Moreover, we ask that all maximal chains in the partial order have exactly one element of each rank. 

Note that this ranking of the faces rescues the concept of dimension; the faces of ranks $0$, $1$ and $n-1$ are called vertices, edges and facets, respectively. A face of rank $i$ is said to be an $i$-face of $\po$.

The maximal chains in the partial order $(\mathcal{P},\leq)$ are called {\em flags. }
The set of flags of a polytope will be denoted by $\mathcal{F}(\mathcal{P})$, and if $\Phi \in \fl(\po)$, then $(\Phi)_i$ shall denote the $i$-face of $\Phi$. 
Two flags that differ in a unique face are said to be {\em adjacent} flags.  

We also require that $\mathcal{P}$ be {\em strong flag connected}, that is, given any two flags $\Phi$ and $\Psi$, there is a sequence of adjacent flags $\Phi= \Phi_0, \Phi_1, ..., \Phi_k = \Psi $ such that $\Phi\cap \Psi \subset \Phi_i$,  for every $0\leq i \leq k$. 
Intuitively this means that given two flags in the polytope, one can ``walk" from one to the other by making changes only in the faces that are not common to both flags. 

Finally, for $\mathcal{P}$ to be an (abstract) polytope we require it to satisfy 
 the {\em diamond condition}, namely, if $i\in [n]$ and $E,F\in \mathcal{P}$ are faces of rank $i-1$ and $i+1$, respectively, with $E< F$, then there exist exactly two faces of $\mathcal{P}$ of rank $i$ that are greater than $E$ and smaller than $F$. 
 
Given $i\in\n$ and a flag $\Phi$, we can deduce from the diamond condition that there is a unique flag that differs from $\Phi$ only at the face of rank $i$. We denote such flag by $\Phi^i$ and say that the flags $\Phi$ and $\Phi^i$ are {\em $i$-adjacent}. 
For the sake of simplicity, we abbreviate  $(\Phi^i)^j$ by $\Phi ^{i,j}$. 
From the diamond condition it is easy to see that given $\Phi \in \mathcal{F}(\mathcal{P})$, we have that $\Phi ^{i,i} = \Phi$, for every $i\in [n]$. 
Also observe that if $ i,j \in\n$, with $|i-j| > 1$, then $\Phi ^{i,j} = \Phi^{j,i}$. 
It is certainly not true, in general, that the equality $\Phi^{i,i+1} = \Phi^{i+1, i}$ holds. 

Given faces $F, G \in \po$, with $F\leq G$, the interval $\{H \in \po \mid F\leq H\leq G\}$ is called a {\em section} of $\po$ and will be denoted by $G/F$. Then $G/F$
 it is a polytope on its own right.

Every polytope has several interesting graphs associated to it.
In this paper, a lot of attention will be given to the {\em flag graph} of a polytope, that we define as follows.
Given a $n$-polytope $\mathcal{P}$, its flag graph $\mathcal{G}_{\mathcal{P}}$ is the graph whose vertex set is the set of flags $\fl(\po)$ 
and whose edges are coloured with colours indexed in $[n]$, in such a way that between two vertices there is an edge of colour $i$ if and only if the two associated flags are $i$-adjacent. 

We can deduce some properties of $\mathcal{G}_{\mathcal{P}}$ from the properties of $\po$. 
For example, the flag connectivity implies that the graph $\mathcal{G}_{\mathcal{P}}$ is connected. 
In fact, given two flags $\Phi, \Psi \in \fl(\po)$, if $\{i_1, i_2, \dots, i_k\}$ is the subset of $\n$ for which $(\Phi)_{i_j}=(\Psi)_{i_j}$, then, the strong flag connectivity of $\po$ implies that, in $\G_\po$, there is a walk from $\Phi$ to $\Psi$ with edges of colours in $\overline{\{i_1, i_2, \dots, i_k\}}$.

Consider now the subgraph formed by taking all the vertices of $\mathcal{G}_{\mathcal{P}}$ and only the edges of colours $i$ and $j$. Observe that this subgraph is a union of disjoint cycles with edges of alternating colours $i$ and $j$. 
Whenever $|i-j| > 1$ all these cycles have length $4$.

The flag graph of a polytope completely determines it. 
That is, two polytopes are isomorphic (as ranked posets) if and only if their corresponding flag graphs are isomorphic (as coloured graphs).
In \cite{STG}, flag graphs of polytopes were used to study symmetry properties of the polytopes.

\subsection{Maps}

Polytopes of rank $3$ can be regarded as maps.
A {\em map}  is a cellular embedding of a connected graph into a surface without boundary, in the sense that the complement of the image of the graph is a collection of disjoin discs. 
These discs are called the {\em faces} of the map.

Let $\mathcal{BS}(\m)$ be the barycentric subdivision of $\m$ and consider a triangle $\Phi$ of $\mathcal{BS}(\m)$.
Label the vertices of $\Phi$ by $\Phi_0$, $\Phi_1$ and $\Phi_2$ according to whether they represent, respectively, the vertex, the edge, or the face of $\Phi$. 
Note that $\Phi$ is adjacent with three other triangles of $\mathcal{BS}(\m)$, each of them having exactly two of the vertices $\Phi_i$. If $\Psi$ is a triangle adjacent to $\Phi$ which does not have the vertex $\Phi_i$ of $\Phi$, we say that $\Phi$ and $\Psi$ are $i$-adjacent and denote $\Psi$ by $\Phi^i$.

With this in mind, it is easy to see a map as an edge-coloured graph $\G_\m$.
Indeed, by taking one vertex per triangle of $\mathcal{BS}(\m)$ and joining two of them by an edge of colour $i$ whenever they are $i$-adjacent, we have constructed an edge-coloured graph. 
This new graph is $3$-regular and  each vertex has an edge of each of the colours $0,1,$ and $2$. 
Further, since each edge of $\m$ belongs to 4 triangles of $\mathcal{BS}(\m)$, then if we remove all edges of colour $1$ from $\G_\m$, we obtain a collection of disjoin $4$-cycles with edges with alternating colours $0$ and $2$.

Note that non-isomorphic maps induce non-isomporhic coloured graphs. 
However, permuting the colours of the edges of $\G_\m$  does change (in general) the map. 
For example, if one interchanges the edges of colour $0$ with those of colour $2$, one gets the {\em dual} map $\m^*$ of $\m$.

\subsection{Edge-coloured graphs}

The flag graphs of maps and polytopes have many features in common, the more obvious one being that they are simple $n$-regular graphs and its edges can be coloured with $n$ colours in such a way that edges incident to one vertex have different colours. (Recall that a simple graph has no loops or multiple edges.) 
In this section we shall see some straightforward properties and give notation of graphs having these two properties.

An {\em edge colouring} of a graph $\G$ is an assignment of colours to the edges of $\G$ such that adjacent edges have different colours. If the number of colours used in an edge colouring of $\G$ is $n$, we say that it is an $n$-edge colouring. The minimal number $n$ such that there exists an $n$-edge colouring of $\G$ is called the {\em chromatic index} of $\G$.
An $n$-regular graph (in the sense that all its vertices have degree $n$) that has chromatic number $n$, together with an $n$-edge colouring is said to be a {\em properly $n$-coloured graph}.
Clearly, a vertex of a properly $n$-coloured graph has one edge of each of the $n$ colours.

Let $\G$ be a properly $n$-coloured graph and suppose that the colouring is given by the colours from a set $C$. 
For each subset $A\subset C$, we denote by $\G_A$ the graph that has all the vertices of $\G$ and edges of colours only in $A$.
In particular, if $A=\{c\}$, for some $c\in C$, then $\G_A$ is simply the perfect matching of $\G$ with edges of colour $c$. 
In this case, we often abbreviate and simply write $\G_c$ for $\G_A$. 
Similarly, if $A=C\setminus\{c\}$, then $\G_A$ will be written as $\G_{\bar{c}}$.

Whenever $|A|=i$, we say that the connected components of $\G_A$ are the $i$-factors of colours in $A$.
Note that since $\G$ is a properly $n$-coloured graph, then the $2$-factors of colours $i$ and $j$, for every $i\neq j \in C$, are cycles with edges of alternating colours $i$ and $j$.

\subsection{Maniplexes}

Maniplexes generalize (the flag graphs of) polytopes and maps at the same time.
In \cite{mani}, where they were first introduced, several equivalent definitions of maniplexes are given. 
Here, we take the most simple approach and define them as edge-coloured graphs in the following way.

An {\em $n$-maniplex} is a properly $n$-coloured simple graph, with edges of colours from $[n]$ such that the $2$-factors of colours $i$ and $j$ are $4$-cycles, whenever $|i-j|>1$.

The vertices of a maniplex are called {\em flags}, and if two flags $u$ and $v$ are adjacent by an edge of colour $i$, then they are said to be {$i$-adjacent}. In that case we often write $v$ as $u^i$ (and $u $ as $v^i$).

For each $i\in[n]$, the connected components of $\m_{\bar{i}}$ are the {\em $i$-faces} of $\m$.
The fact that  the $2$-factors of colours $i$ and $j$ are $4$-cycles, whenever $|i-j|>1$, implies the following lemma.

\begin{lemma} \label{FacesofM}
Let $\m$ be an $n$-maniplex. Let $F$ be an $i$-face of $\M$, with $i \in \{1, 2, \dots , n-2\}$. Then $F$ is the cartesian (graph) product of some connected component of $\M_{\{0,\dots, i-1\}}$ by some connected component of $\M_{\{i+1,\dots, n-1\}}$.
\end{lemma}

Note that for $n\leq 2$, flag graphs of polytopes and maniplexes coincide. A $0$-maniplex consists of a unique flag, which is also the $0$-face of the maniplex.  A $1$-maniplex has two flags and an edge of colour $0$ joining them. It has two $0$-faces. A $2$-maniplex is either a $2p$-cycle or an infinite path, with alternating edges of colours $0$ and $1$. The $0$-faces have two flags, and so do the $1$-faces. If the 2-maniplex is a $2p$-cycle, then it corresponds to the graph flag of an (abstract) $p$-gon, and if it is an infinite cycle, corresponds to the graph flag of an infinite abstract $2$-polytope.

\section{Maniplexes as posets}\label{posets}

Whenever a maniplex $\m$ is (isomorphic to) the flag graph of a polytope, we shall say that $\m$ is {\em polytopal}.
It is well-known that not every map can be regarded as a polytope. 
Hence, not every maniplex is polytopal.

Take, for example, the map on the torus usually denoted by $\{4,4\}_{(1,1)}$. 
This map is the quotient of the square tessellation of the plane, with vertices in $\mathbb{Z} \times \mathbb{Z}$, by the vectors $(1,1)$ and $(-1,1)$ (see Figure~\ref{(4,4)_(1,1)}). 
It is a $3$-maniplex with 16 flags, 2 vertices, 4 edges and 2 faces.
A quick inspection of the map shows that every vertex is incident to every edge and to every face. 
Thus, between a vertex and a face there are 4 edges and the diamond condition is not satisfied.
\begin{figure}[htbp]
\begin{center}
\includegraphics[width=5cm]{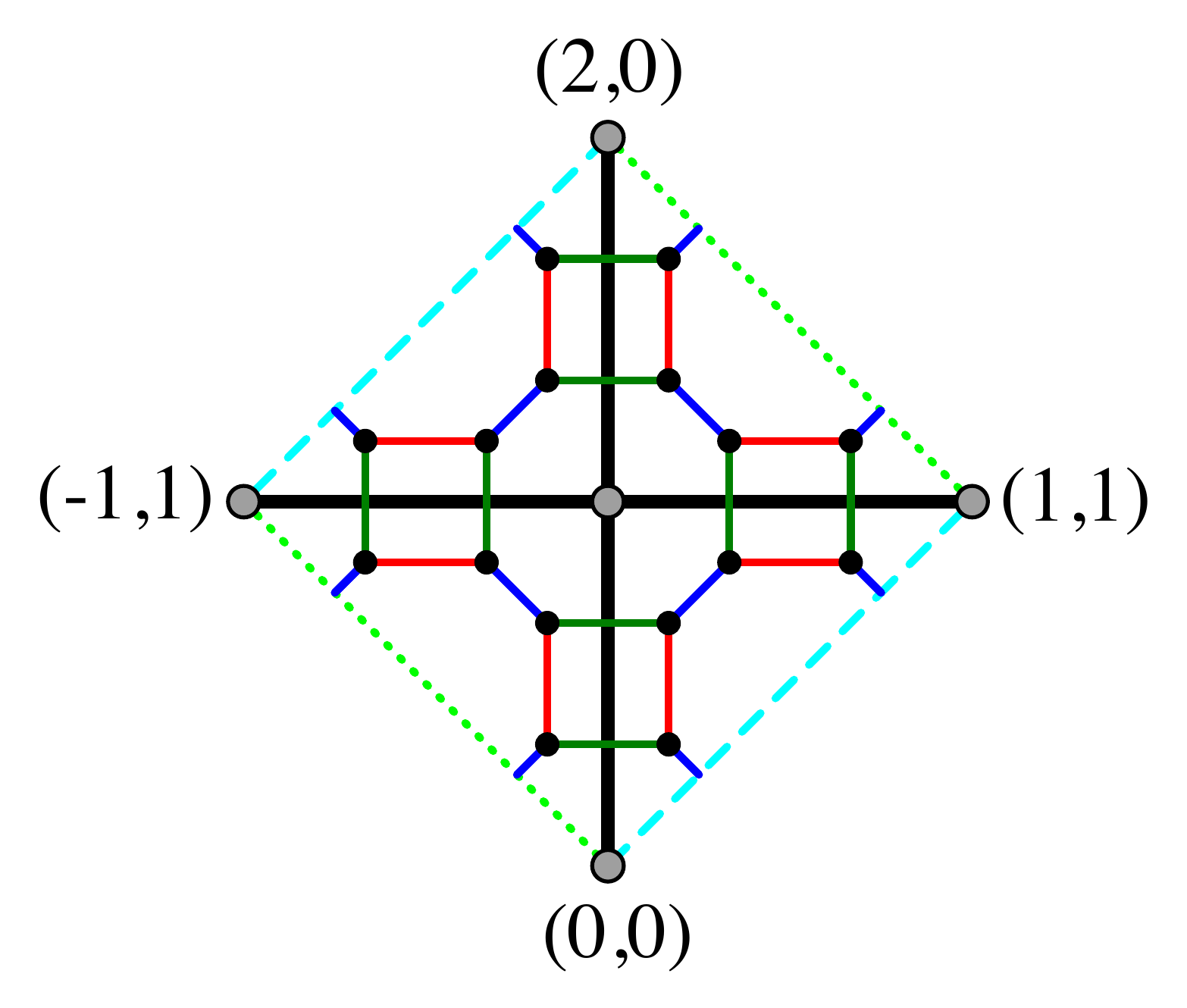}
\caption{The map $\{4,4\}_{(1,1)}$ on the torus and its graph flag.}
\label{(4,4)_(1,1)}
\end{center}
\end{figure}

Although the map considered above is not a polytope, their faces can be seen as a poset and such poset possesses all the combinatorial information of the map. 
The Hasse diagram of this poset can be seen in Figure~\ref{poset(4,4)_(1,1)}. 
It is straightforward to see that this poset, just as the map, has 16 flags, 2 vertices, 4 edges and 2 faces.
\begin{figure}[htbp]
\begin{center}
\includegraphics[width=2.5cm]{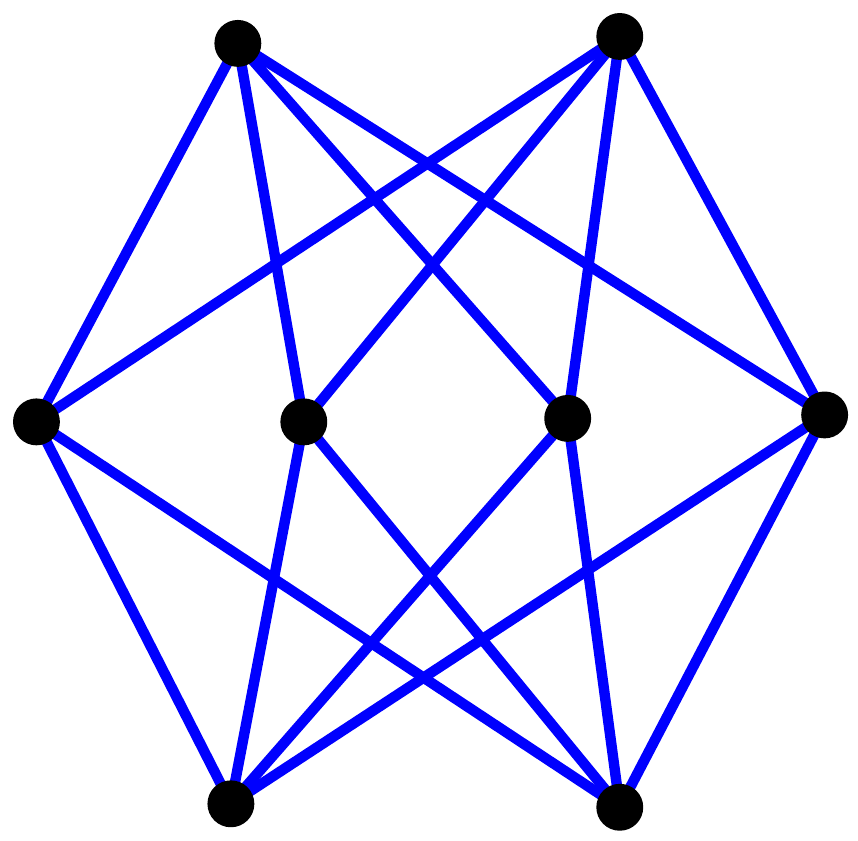}
\caption{Hasse diagram of the poset resulting from the map $\{4,4\}_{(1,1)}$.}
\label{poset(4,4)_(1,1)}
\end{center}
\end{figure}

Every maniplex $\M$ can be seen as a ranked poset in the following way.
The elements of the poset are the faces of $\M$, and the rank of an $i$-face of $\M$ is precisely $i$.
Recall that the $i$-faces of $\M$ are the connected components of $\M_{\bar{i}}$.
Given an $i$-face $F_i$ and a $j$-face $F_j$ of $\M$, we shall say that
\begin{eqnarray} \label{order}
F_i \leq F_j \ \mathrm{if \ and \ only \ if} \ i\leq j \ \mathrm{and} \ F_i\cap F_j \neq \emptyset.
\end{eqnarray}
Notice that we denote in the same way the connected components of each $\m_{\bar{i}}$ and the corresponding elements of the ordered set. 
Depending on the context it will be clear if we are talking about a graph or an element of a poset.

\begin{prop}\label{prop:order}
Let $\M$ be a maniplex. The set of faces of $\m$, together with the relation
 ``$\leq$'' as defined in (\ref{order}), is a poset.
\end{prop}

\begin{proof}
Since the faces of $\M$ are connected components of the graph, the definition of $``\leq''$ immediately implies that it is reflexible and anti-symmetric. 
We only need to show the transitivity.
Let $E,F$ and $G$ be faces of $\M$ such that $E\leq F$ and $F\leq G$ and let $v\in E\cap F$ and $u\in F\cap G$. 

It is clear that if the ranks of $E,F,G$ are $i,j,k$, respectively, then $i\leq j \leq k$. 
Thus, we only need to show that $E\cap G \neq \emptyset$. 

We shall  find a flag $w \in F$ such that there is a path with edges of colours $0,\dots, j-1$ from $u$ to $w$ and another path with edges of colours $j+1,\dots, n-1$ form $v$ to $w$. 
It is straightforward to see that such $w$ is then an element of $E \cap G$.

Suppose that $E\neq F \neq G$, as otherwise $E\cap G \neq \emptyset$ trivially. 
Hence, $i< j < k$ and thus $0<j<n-1$. 
By Lemma~\ref{FacesofM}, the $j$-face $F$ is the product of two graphs, say $F_<$ and $F_>$ such that $F_<$ is a connected component of $\M_{\{0,\dots, j-1\}}$ and $F_>$ is a connected component of $\M_{\{j+1, \dots, n-1\}}$.
Hence, we can write $v=(v_0,v_1)$, $u=(u_0,u_1)$, with $v_0, u_0 \in F_<$ and $v_1, u_1 \in F_>$. The vertex $w=(v_0,u_1)\in F$ has the required property.
\end{proof}

Let $\po_\m$ be the set of all faces of $\m$, together with a minimum and a maximum element, denoted by $F_{-1}$ and $F_n$, respectively.
It is clear that  $\po_\m$ is a ranked poset. 
The ranks of $F_{-1}$ and $F_n$ are set to be $-1$ and $n$, respectively. These two new faces are the {\em improper} faces of $\po_\m$.  

It is not difficult now to show that all the maximal chains of $\po_\m$ have $n+2$ elements. 
Suppose that $E$ and $G$ are two incident faces of ranks $i$ and $k$, respectively, with $i<k-1$, and let $w \in E \cap G$. 
For each $j\in \{i+1, \dots, k-1\}$ consider $F_j$, the connected component of $\m_{\overline{j}}$ containing $w$. 
Since $w$ has one edge of each colour, such $F_j$ clearly exists.
By definition, $F_j$ is a $j$-face of $\m$, and $w\in E\cap F_j\cap G$ implying that both $E$ and $G$ are incident to $F_j$.

The idea of the proof of Proposition~\ref{prop:order} can be extended to show the following lemma. 

\begin{lemma}\label{lemma:nonemptychains}
Let $\{F_1, F_2, \dots F_k\}$ be a chain of proper faces of $\po_\m$. Then $\bigcap_{j=1}^k F_{j}$ is non-empty.
\end{lemma}

Although Proposition~\ref{prop:order} says that given a maniplex, we can define an order on its faces simply by the non-empty intersection of them, it does not mean that the defined poset possesses much information about the maniplex. 
To clarify this, we give another example of a map (i.e. a 3-maniplex) that is not the flag graph of a polytope.
Consider the map on the torus $\m=\{4,4\}_{(1,0)}$.
The maniplex corresponding to this map and the Hasse diagram of the poset $\po_\m$ are given in Figure~\ref{fig:(4,4)_(1,0)}.
The maniplex has 8 flags, 1 vertex (of degree 4), 2 edges and 1 face (a square). Clearly, it does not satisfy the diamond condition.
Furthermore, the induced poset $\po_\m$ has only two maximal chains. 
Hence, given only the poset, we know how many $i$-faces the maniplex has, but
we cannot obtain all information about the original maniplex. 
In fact, the maniplex in Figure~\ref{Klein} has the same induced order (on the left in Figure~\ref{fig:(4,4)_(1,0)}). 
These two maniplexes (in Figures~\ref{fig:(4,4)_(1,0)} and~\ref{Klein}) are different. In particular, one induces a map on the torus and the other a map on the Klein Bottle.

\begin{figure}[htbp]
\begin{center}
\includegraphics[width=7cm]{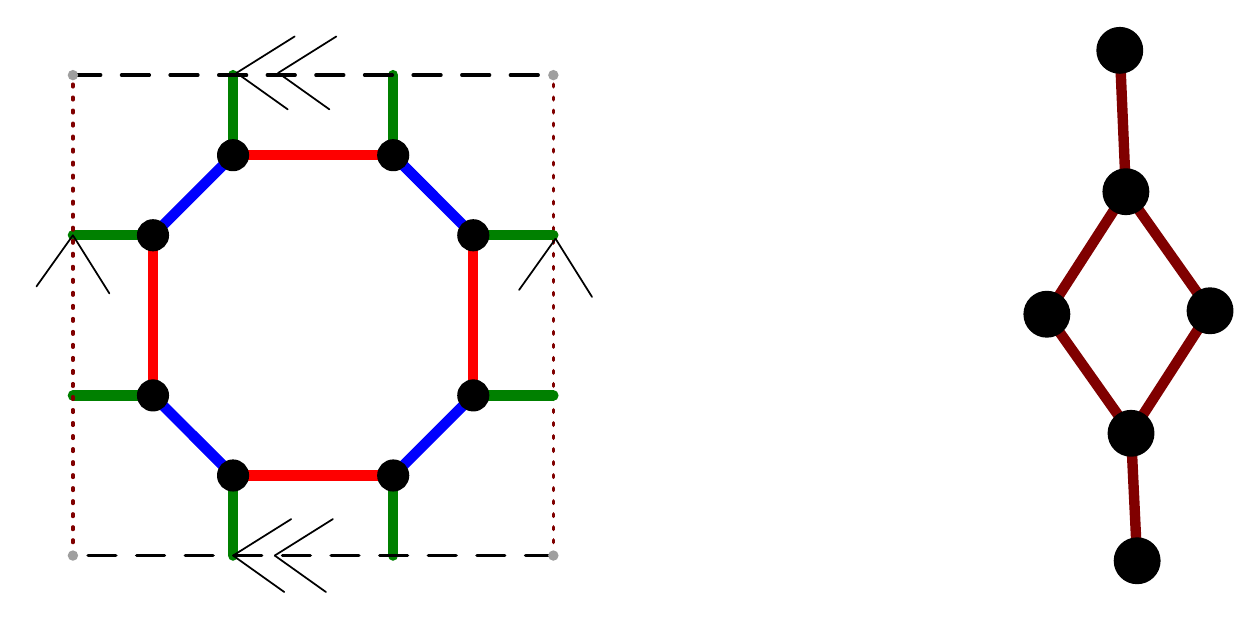}
\caption{The map $\M=\{4,4\}_{(1,0)}$ on the torus and the Hasse diagram of $\po_\m$.}
\label{fig:(4,4)_(1,0)}
\end{center}
\end{figure}

\begin{figure}[htbp]
\begin{center}
\includegraphics[width=3cm]{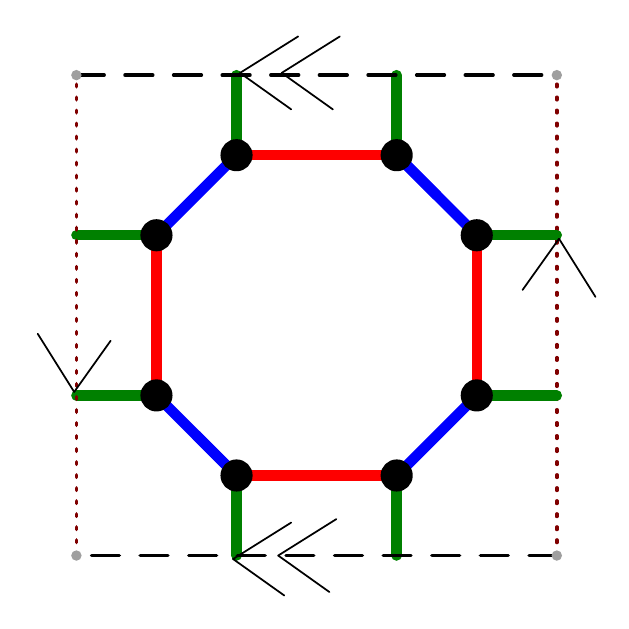}
\caption{A map on the Klein Bottle.}
\label{Klein}
\end{center}
\end{figure}

Whenever there is a bijection between the maximal chains of the poset $\po_\m$ induced by the maniplex $\m$ and the flags of $\m$, we say that $\po_\m$ is {\em faithful}.
In general, flags of maniplexes might not be completely determined by the faces they are contained in. That is, there can be different flags contained in exactly the same faces (as in the above example). However, when the induced poset is faithful, then flags of the maniplex are completely determined by their faces. 
In fact, in this case given a maximal chain $\Phi=\{F_0, F_1, \dots, F_{n-1}\}$ of the induced poset, then $\Phi$ represents the unique flag of $\m$.
In other words, $\po_\m$ is faithful if and only if for every maximal chain $\{F_0, F_1, \dots, F_{n-1}\}$ of $\po_\m$, we have that $\bigcap_{i=0}^{n-1} F_i$ consists of a unique vertex.

\section{The component intersection property}\label{sec:CIP}
When one studies 3-maniplexes (or maps) that are not polytopal, it is easy to get the idea that the problem with polytopality is the diamond condition. In this section, we hope to convince the reader that this is far from the truth, and that, philosophically speaking, the only thing that can really fail is the strong connectivity. 
In fact, we shall see that if the poset corresponding to a given maniplex is faithful and strongly flag connected, then it satisfies the diamond condition.
Before getting into that discussion, we shall analyze the meaning of the strong flag connectivity of a polytope in terms of its flag graph.

Let $\po$ be a polytope and $\m=\G_\po$ be its flag graph. 
Let $\Phi, \Psi \in \fl(\po)$ and let $v_\Phi,v_\Psi \in \G_\po$ be the corresponding vertices.
That is, we think $\Phi$ as a maximal chain of the order $\po$ and $v_\Phi$ as the vertex of $\G_\po$ induced by $\Phi$.
As we pointed out in Section~\ref{sec:poly}, the strong flag connectivity of $\po$ implies that there is a 
path from $v_\Phi$ to $v_\Psi$ that has all its edges of colours different to the ranks of the faces in $\Phi\cap\Psi$.

Let us see what this means in terms connected components of $\G_\po$. 
Suppose that $\Phi\cap\Psi = \{F_{1}, F_{2}, \dots F_{k}\}$, where $F_{j}$ has rank $i_j\in\{0,1,\dots, n-1\}$.
This means that for each $j=1,\dots, k$,  $F_{j}$ can be regarded as a connected component of $\m_{\overline{i_j}}$ containing both $v_\Phi$ and $v_\Psi$. 
Hence, $v_\Phi,v_\Psi\in \bigcap_{j=1}^k F_{j}$.
The strong flag connectivity of $\po$ implies that there is a path from $v_\Phi$ to $v_\Psi$ with no edge of colour $i_1, i_2, \dots, i_k$. 
Therefore, the path is contained in $F_{j}$, for each $j=1,\dots k$ and thus it is contained in $\bigcap_{j=1}^k F_{j}$.

Now let $\m$ be an $n$-maniplex, and let $\{i_1, i_2, \dots i_k\}$ be a subset of $\n$ with $i_1<i_2<\dots <i_k$.
Consider a set of faces $G_1, G_2, \dots G_k$, where $G_j$ is a $i_j$-face, such that $G_j\cap G_{j+1} \neq \emptyset$ for all $j$.
Then, $\{G_1, G_2, \dots G_k\}$ can be regarded as a chain of the poset $\po_\m$ and $\bigcap_{j=1}^k G_{j}$ is a subgraph of $\m_{\overline{i_1, i_2, \dots, i_k}}$, that may have several components.

\begin{definition}
Let $\m$ be a $n$-maniplex and let $\po_\m$ be the poset defined by its faces. We say that $\m$ has the {\em component intersection property (or CIP)} if for every chain $\{G_1, G_2, \dots G_k\}$ of the poset $\po_\M$ we have that $\bigcap_{j=1}^k G_{j}$ is connected.
\end{definition}

Recall that in the above definition $G_i$ stands for both an element of the poset $\po_\m$ and a subgraph of $\m$. The intersection $\bigcap_{j=1}^k G_{j}$ is, then, considered as a subgraph of $\m$.

Although we use $\po_\m$ to define the CIP in a maniplex, this is not strictly necessary. We could state the definition in terms of connected components of the subgraphs $\m_{\overline{i}}$ and their non empty intersection. However, using $\po_\m$ helps us to have a cleaner definition. 

\noindent {\bf Fact:} Not every maniplex has the CIP. 

One can see that in the map $\M=\{4,4\}_{(1,1)}$  the intersection of a connected component of $\m_{\{0,1\}}$ with a connected component of $\m_{\{1,2\}}$ consists of two edges of colour $1$ (Figure~\ref{CIP{4,4}(1,1)}). Also, in the map $\m=\{4,4\}_{(1,0)}$  the intersection of a connected component of $\m_{\{i,j\}}$ with a connected component of $\m_{\{j,k\}}$ consists of either two or four edges of colour $k$ (where $\{i,j,k\}=\{0,1,2\}$) (Figure~\ref{CIP{4,4}(1,0)}).
\begin{figure}[htbp]
\begin{center}
\includegraphics[width=11cm]{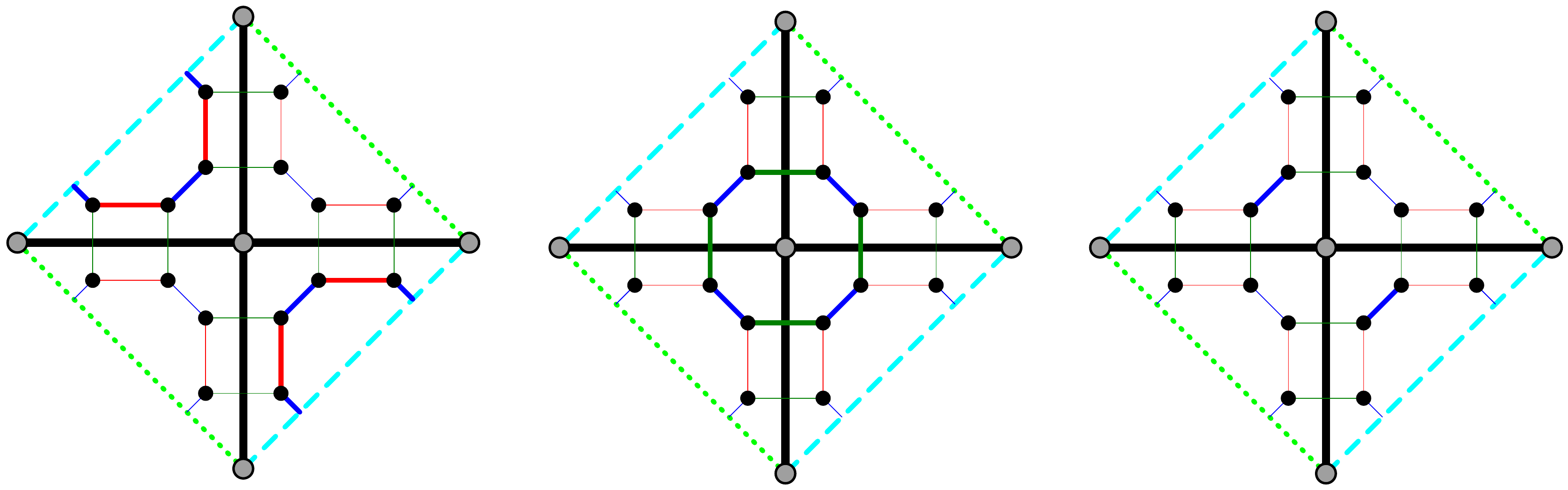}
\caption{A connected component of $\m_{\{0,1\}}$ and a  connected component of $\m_{\{1,2\}}$, whose intersection is not connected.}
\label{CIP{4,4}(1,1)}
\end{center}
\end{figure}

\begin{figure}[htbp]
\begin{center}
\includegraphics[width=9.5cm]{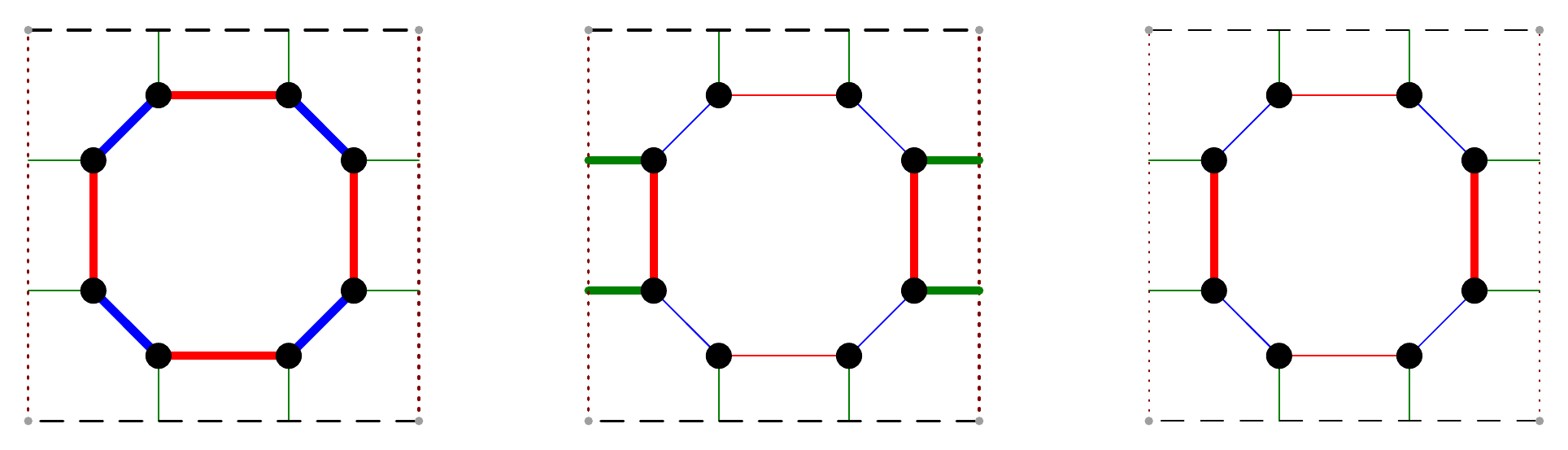}
\caption{A connected component of $\m_{\{0,1\}}$ and a  connected component of $\m_{\{0,2\}}$, whose intersection is not connected.}
\label{CIP{4,4}(1,0)}
\end{center}
\end{figure}

The discussion leading to the definition of the CIP
tells us that the strong connectivity of $\po_\m$ and the CIP are properties that are closely related In fact, we have the following result.

\begin{lemma}\label{SFC<->CIP}
Let $\m$ be a maniplex such that its induced poset $\po_\m$ is faithful. Then, $\po_\m$ is strongly flag connected if and only if $\m$ has the CIP.
\end{lemma}

\begin{proof}
Suppose that $\po_\m$ is strongly flag connected and let $\{G_1, G_2, \dots G_k\}$ be a chain of $\po_\m$. We need to show that $\bigcap_{j=1}^k G_{j}$ is connected, so let $v, u \in \bigcap_{j=1}^k G_{j}$. Consider $\Phi_v, \Phi_u$ the (unique) maximal chains of $\po_\m$ corresponding to $v$ and $u$, respectively. Since $\po_\m$ is strongly flag connected, then there is a sequence of adjacent flags for $\Phi_v$ to $\Phi_u$ all containing the faces $\{G_1, G_2, \dots G_k\}$. This sequence induces a path in $\m$ that starts at $v$ and that travels in $\bigcap_{j=1}^k G_{j}$. The fact that $\po_\m$ is faithful implies that the induced path finishes at $u$, and thus that $\bigcap_{j=1}^k G_{j}$ is connected. (Note that if $\po_\m$ is not faithful, but it is flag connected, then the sequence of adjacent flags from $\Phi_v$ to $\Phi_u$ starts at $v$ but might finish at a vertex $u'\neq u$ which is contained in exactly the same faces than $u$.)

Now suppose that $\m$ has the CIP and let $\Phi$ and $ \Psi$ be two maximal chains of $\po_\m$.
Then, $\Phi=\{F_{-1}, F_0, F_1, \dots, F_{n-1}, F_n\}$ and $\Psi=\{F_{-1}, G_0, G_1, \dots ,G_{n-1}, F_n\}$, where $F_i$ and $G_i$ are $i$-faces of $\m$.
Let $\{i_1, i_2, \dots ,i_k\}$ be the subset of $ \n$ containing all the indices satisfying that $F_{i_j}=G_{i_j}$.

By Lemma~\ref{lemma:nonemptychains}, $\bigcap_{i=0}^{n-1} F_{i} \neq \emptyset \neq \bigcap_{i=0}^{n-1} G_{i}$.
Let $v\in\bigcap_{i=0}^{n-1} F_{i}$ and $u \in \bigcap_{i=0}^{n-1} G_{i}$.
Hence, $v,u\in \bigcap_{j=1}^k F_{i_j} = \bigcap_{j=1}^k G_{i_j}$. 
Since $\m$ has the CIP, then $\bigcap_{j=1}^k F_{i_j}$ is connected implying there is a path from $v$ to $u$ with
all its edges of colours in $\overline{\{i_1, i_2, \dots i_k\}}$.
The path from $v$ to $u$ in $\m$ defines a sequence of adjacent flags from $\Phi$ to $\Psi$ and the fact that all the edges of the path have colours in  $\overline{\{i_1, i_2, \dots, i_k\}}$ implies that the adjacencies between the flags of the sequence are all in  $\overline{\{i_1, i_2, \dots ,i_k\}}$. Therefore $\po_\m$ is strongly flag connected.
\end{proof}

We now turn our attention to the diamond condition.
The diamond condition of $\po_\m$ has no effect on the fact that $\m$ satisfies the CIP. 
If $\po$ is a polytope, then it has the diamond condition and $\G_\po$ satisfies the CIP. However,
consider the following example, that also has the diamond condition, but does not satisfy the CIP.

Let $\cal T$ be the tessellation of Euclidean 3-space by cuboctahedra and octahedra. 
This tessellation can be obtained from the cube tessellation of Euclidean 3-space by fully truncating each of the cubes. Assuming that the vertices of the cube tessellation (before truncating) coincide with the integer lattice $\mathbb{Z}^3$, consider the vectors $v_1=(0,2,0)$, $v_2=(1,0,0)$ and $v_3=(1,0,2)$.
Then quotient the tessellation $\cal T$ by the translation group $\Lambda$ generated by $v_1, v_2$ and $v_3$.
We then obtain a tessellation of the $3$-torus by cuboctahedra and octahedra (see Figure~\ref{fig:TT}, where vertices of the same colour are identified). 
This tessellation can be regarded as a maniplex and as a poset.
It is not difficult to see that the induced poset satisfies the diamond condition (one can use the symmetry to check only few cases), however, the CIP is not satisfied in the induced maniplex.
This can be seen by taking the green vertex labelled $A$ in Figure~\ref{fig:TT} and the cuboctahedron at the upper right corner (of the figure). 
These faces are incident to each other, however the corresponding section of the induced poset is not connected.
In fact, the $3$-maniplex induced by the vertex is isomorphic to a cube (where the connections have colours $1$, $2$, and $3$), and the intersection of such $3$-maniplex and the maniplex induced by the cuboctahedron is the disjoin union of two $8$-cycles with edges of alternating colours $1$ and $2$.
\begin{figure}[htbp]
\begin{center}
\includegraphics[width=6cm]{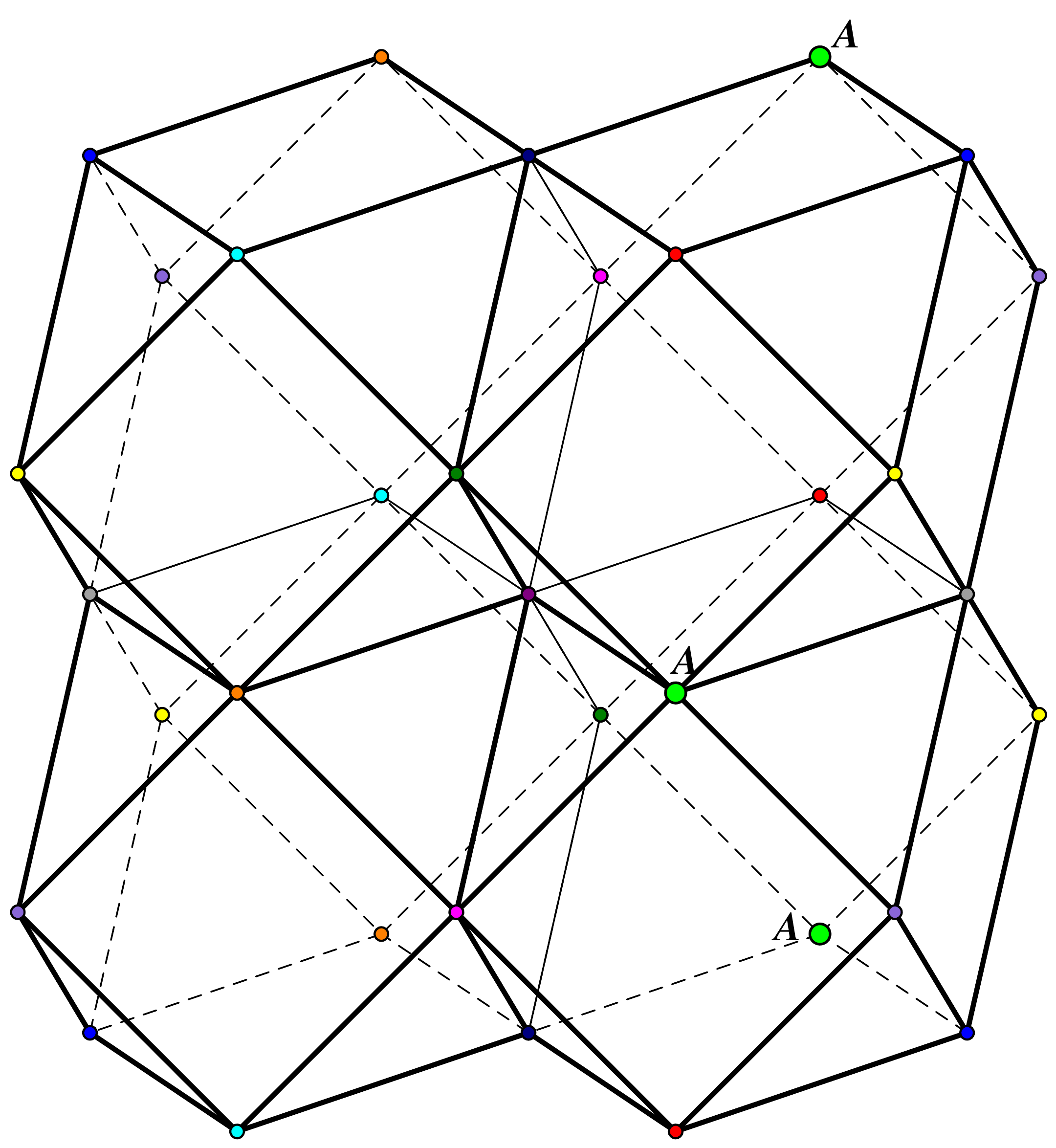}
\caption{By identifying vertices with the same colours, we obtain a $3$-torus tessellated by 4 cuboctahedra and 4 octahedra that satisfies the diamond condition, when considered as a poset. The tessellation induces a maniplex that is not polytopal.}
\label{fig:TT}
\end{center}
\end{figure}

We have seen that not all maniplexes have the CIP. 
Even those whose induced poset satisfy the diamond condition might fail to have the CIP.
However,  maniplexes that are the flag graphs of polytopes satisfy the CIP. 
We now show that in fact having the CIP is a sufficient condition for $\po_\m$ to have the diamond condition.

\begin{lemma}\label{CIP->POL}
Let $\m$ be a $n$-maniplex having the CIP. Then $\po_\m$ satisfies the diamond condition.
\end{lemma}

\begin{proof}
Let $F_{i-1}$ and $F_{i+1}$ be two incident faces of $\po_\m$ having ranks $i-1$ and $i+1$, respectively. 
Since all maximal chains have one element of each rank, then there exist faces $F_0,\dots, F_{i-2}, F_{i+2}, \dots, F_{n-1}$ such that $\{F_{-1}, F_0, \dots, F_{i-1}, F_{i+1}, \dots, F_{n-1}, F_n\}$ is a chain of $\po_\m$ having elements of all ranks but of rank $i$.
By Lemma~\ref{lemma:nonemptychains}, $\bigcap_{j\neq i} F_j \neq \emptyset$ and since $\m$ has the CIP, then $\bigcap_{j\neq i} F_j$ is connected.
As each $F_j$ is a connected component of $\m_{\overline{j}}$, we know that $\bigcap_{j\neq i} F_j$ is a connected component of $\m_{i}$.
But the $\m_{i}$ is the perfect matching of $\m$ whose edges are all of $i$. That is, $\bigcap_{j\neq i} F_j$ is an edge of colour $i$. 
Thus, there are exactly two flags of $\m$ to which one can extend the chain $\{F_{-1}, F_0, \dots ,F_{i-1}, F_{i+1}, \dots, F_{n-1}, F_n\}$ and the diamond condition is satisfied.
\end{proof}

\begin{coro}
If $\m$ is a maniplex such that $\po_\m$ is faithful and strongly flag connected, then $\po_\m$ satisfies the diamond condition (and hence it is a polytope).
\end{coro}

In other words,  if a maniplex is not polytopal then its induced poset fails to be either faithful or strongly flag connected (and it might or might not satisfy the diamond condition). The following theorem implies that the polytopality and the CIP are equivalent conditions.

\begin{theorem}\label{CIP<->Poly}
Let $\m$ be a maniplex and let $\po_\m$ be its induced poset.  $\po_\m$ is a polytope if and only if $\m$ satisfies the CIP. Moreover, in such case $\m$ is isomorphic to the flag graph of $\po_\m$.
\end{theorem}

\begin{proof}
Suppose that $\po_\m$ is a polytope. If $\po_\m$ is faithful, by Lemma~\ref{SFC<->CIP}, we are done. The poset $\po_\m$ is faithful if for every maximal chain $\{F_0, F_1, \dots, F_{n-1}\}$ of $\po_\m$, the intersection $\bigcap_{i=0}^{n-1} F_{i}$ consists of a unique vertex of $\m$. In such case, the bijection between the maximal chains of $\po_\m$ and the vertices of $\m$ is given in a natural way.

We now show that if $\po_\m$ is a polytope, then for every maximal chain $\{F_0, F_1, \dots, F_{n-1}\}$ of $\po_\m$, we have that $|\bigcap_{i=0}^{n-1} F_{i}| = 1.$
The proof is done by induction over $n$, the rank of the polytope $\po_\m$. For $n=0,1,2$, the proposition follows immediately, as the faces $F_i$ are simply vertices or edges. 
Suppose that the proposition is true for every rank less than $n$, and that $\po_\m$ has rank $n$. 
Let $\{F_0, F_1, \dots, F_{n-1}\}$ be a maximal chain of $\po_\m$. Then $\bigcap_{i=0}^{n-1} F_{i}\neq \emptyset$

The face $F_{n-1}$ is a connected component of $\m_{\overline{n-1}}$, which implies that it is a maniplex of its own right. Thus, $F_{n-1}$ induces a partial order $\po_{n-1}:=\po_{F_{n-1}}$. We shall show that $\po_{n-1}$ is isomorphic to the section $F_{n-1} / F_{-1}$ of $\po_\m$.

For each $v \in F_{n-1}$, consider $\Phi_v$ the corresponding maximal chain of $\po_\m$. By the diamond condition of $\po_\m$, there is a flag $\Phi_v^{n-1}$ that coincides with $\Phi_v$ in all its faces of ranks $0, \dots, n-2$, but has some $F'_{n-1}\neq F_{n-1}$ as its $(n-1)$-face. Then, $v^{n-1}$ , the $(n-1)$-adjacent flag to $v$ is in $F'_{n-1}$, and hence it does not belong to $F_{n-1}$. 
This implies that, if $G_i$ is an $i$-face of $\m$ such that $G_i \cap F_{n-1}\neq \emptyset$  (i.e $G_i \leq F_{n-1}$ in $\po_\m$), then $G_i \cap F_{n-1}$ is an $i$-face of the $(n-1)$ maniplex $F_{n-1}$ and hence it is an $i$-face of the induced order $\po_{n-1}$.
In other words, the $i$-faces of the section $F_{n-1} / F_{-1}$ of $\po_\m$ can be thought as $i$-faces of $\po_{n-1}$. 
Moreover, every $i$-face of $\po_{n-1}$ comes from an $i$-face of $F_{n-1} / F_{-1}$ of $\po_\m$. In fact, if $G$ is an $i$-face of $\po_{n-1}$, then $G$ is a connected component of $(F_{n-1})_{\bar{i}}$. Take $v \in G$ and consider $\hat{G}$, the connected component of $\m_{\bar{i}}$ containing $v$.
Then $\hat{G}$ is an $i$-face of $\m$ and since $v\in G \subset F_{n-1}$, then $\hat{G}\cap F_{n-1} \neq \emptyset$, implying that $\hat{G} \leq F_{n-1}$ in $\po_\m$ and hence $G=\hat{G}\cap F_{n-1}$. That is, $G$ comes from the $i$-face $\hat{G}$ of $F_{n-1} / F_{-1}$.
This gives a bijection between the faces of $\po_{n-1}$ and those of $F_{n-1} / F_{-1}$. The fact that they are isomorphic as posets is now straightforward.

Then, $\po_{n-1}$ is a polytope of rank $n-1$. By induction hypothesis, this implies that $\po_{n-1}$ is faithful. 
Then, if for each $i=0, \dots, n-2$ we set $G_i:=F_i\cap F_{n-1}$, we have that $G_i$ is an $i$-face of the $(n-1)$-maniplex $F_{n-1}$, and $G_i$ is an $i$-face of $\po_{n-1}$. Thus,  $\{G_0, G_1, \dots G_{n-2}\}$ is a maximal chain of $\po_{n-1}$. 
Since $\po_{n-1}$ is faithful, then $|\bigcap^{n-2}_{i=0} G_i| = 1$. 
But $\bigcap^{n-2}_{i=0} G_i =\bigcap^{n-2}_{i=0} (F_i\cap F_{n-1})=\bigcap_{i=0}^{n-1} F_{i}$,
 implying that $|\bigcap_{i=0}^{n-1} F_{i}|=1$. Therefore $\po_\m$ is faithful and thus $\m$ has the CIP.

We now assume that $\m$ satisfies the CIP. 
Then $\po_\m$ is a poset with rank function that has a minimal and a maximal element. By Lemma~\ref{CIP->POL}, $\po_\m$ also satisfies the diamond condition.
If $\po_\m$ is faithful, then by Lemma~\ref{SFC<->CIP} we have that $\po_\m$ is a polytope.

In what follows, we show that there exists a bijection $\beta$ between the flags of $\m$ and the maximal chains $\fl(\po_\m)$ of $\po_\m$.
For each $v \in \m$ and $i \in [n]$, we let $F_i^v$ be the connected component of $\m_{\overline{i}}$ containing $v$. Clearly,  each  $F_i^v$ is well-defined and unique.
Let $\beta: \m \to \fl(\po_\m)$ be such that
\begin{eqnarray*}
\beta:
 v \mapsto \{F_{-1}, F^v_0, F^v_1, \dots, F^v_{n-1}, F_n\} .
\end{eqnarray*}
Then $\beta$ is a well defined function from $\m$ to $\fl(\po_\m)$. 
By Lemma~\ref{lemma:nonemptychains}, $\beta$ is onto.
Now suppose that $v,u\in\m$ are such that $\beta(v)=\beta(u)= \{F_{-1}, F_0, F_1, \dots, F_{n-1}, F_n\}$. 
This means that both $v$ and $u$ are elements of each $F_i$, $i\in[n]$.
Hence, $v,u \in \bigcap_{i=0}^{n-1} F_{i}$.
Note that on one hand, since $\m$ has the CIP,  $\bigcap_{i=0}^{n-1} F_{i}$ is connected. 
On the other hand,
as each $F_i$ fails to have edges of colour $i$,  $\bigcap_{i=0}^{n-1} F_{i}$ has no edges, so it is a set of disjoin vertices. 
Therefore $\bigcap_{i=0}^{n-1} F_{i}$ consists of exactly one vertex of $\m$, implying that $v=u$ and thus $\beta$ is one to one.

Hence, if $\m$ satisfies the CIP, then $\po_\m$ is faithful and therefore it is a polytope. Moreover,
 $\beta$ can be regarded as a bijection between the vertices of $\m$ and the vertices of $\G_{\po_\m}$. 
To see that it is indeed a isomorphism of $\m$ and $\G_{\po_\m}$ (as coloured graphs) one only needs to note that $i$-edges of $\G_{\po_\m}$ correspond to $i$-adjacent flags of $\po_\m$, and these in turn (as pointed out in the proof of Lemma~\ref{CIP->POL}) correspond to $i$-edges of $\m$. 
Theus, the theorem follows.
\end{proof}

A corollary of the above theorem is that a maniplex is polytopal if and only if it satisfies the CIP.

\section{The path intersection property}\label{sec:PIP}

In the previous section we gave necessary and sufficient conditions on a maniplex to be polytopal. 
Such conditions were given in terms of some connected components of the maniplex. 
In this section we show that the CIP is equivalent to an intersection property of the paths of the maniplex. 
By doing this, we will have necessary and sufficient conditions for the polytopality in terms of intersections of coloured paths of the maniplex.

We start by looking into common ways to manipulate a given path of the maniplex without changing its colours. 
We use (several times) that  the 2-factors of colours $i$ and $j$ are $4$-cycles, whenever $|i-j|>1$.
For instance, if we have a path of colours $1, 4, 1$, we can replace it by an edge of colour $4$, or we can assume that whenever there are two consecutive edges on a path whose colours differ in more than one, then the smallest (or the greatest) colour appears first.

It shall prove particularly useful to be able to manipulate paths within a face of the maniplex. 
Suppose that $u,v \in F_i$, where $F_i$ is an $i$-face of a maniplex $\m$.
Of course, this means that there is a path from $u$ to $v$ whose edges are not of colour $i$.
But, can we find a nice path that connects $u$ and $v$?
This depends on what one understand by ``nice'', but we shall see that we can find a path from $u$ to $v$ that passes through a vertex $w$ satisfying the property that the path from $u$ to $w$ has edges of colours all smaller that $i$, while the one from $w$ to $v$ has edges of colours all greater than $i$.

To see this, consider $p$, a path from $u$ to $v$ that is contained in $F_i$.
As pointed out before, with out loss of generality, we can assume that, in $p$, if there are two consecutive edges of colours $j$ and $k$ with $|j-k|>1$, then the edge of smaller colour appears first.
Let $u=v_0, v_1, \dots v_m=v$ be the vertices of $p$ and let $e_j$ be the edge of $p$ connecting the vertices $v_{j-1}$ and $v_j$. 
Let $c(e_j)$ denote the colour of the edge $e_j$.
Note that, by our assumptions, if for some $j=2,\dots,m-1$ we have that $c(e_{j-1})=c(e_{j+1})$, then $c(e_j)=c(e_{j-1})\pm 1$.
Moreover, if $c(e_j)>i$ for some $j>0$, then all the edges $e_{j+1}, \dots, e_{m}$ have colours also greater than $i$.
Indeed if this was not the case, suppose that $e_{j+k}$ is the first edge of $p$ after $e_j$ that has colour smaller than $i$. 
But then $c(e_{j+k-1}) >i$ and $c(e_{j+k})<i$ imply that $|c(e_{j+k-1})-c(e_{j+k})|>1$, and by our assumptions, $e_{j+l}$ appears before $e_{j+l-1}$ in $p$, which is a contradiction.

We have therefore stablished the following lemma.

\begin{lemma}\label{orderedpaths}
Let $\m$ be a maniplex and $u,v \in F_i$, where $F_i$ is an $i$-face of $\m$.
Then, there exists a vertex $w\in F_i$ satisfying the property that there is a path from $u$ to $w$ having edges of colours all smaller that $i$, while there is a path from $w$ to $v$ having edges of colours all greater than $i$.
\end{lemma}

In light of Lemma~\ref{FacesofM}, when $F_i$ is thought as a graph product, we can write $u=(u_1,u_2), v=(v_1,v_2) \in F_i$. Then the $w$ of the above lemma is simply the vertex $w=(v_1, u_2)\in F_i$.

Observe that, if we start by assuming that there is a $u-v$ path in $F_i$ that does not have colours in some set $A\subset\n$, then the $u-w$ and $w-v$ paths given in the above lemma can be taken in such a way that do not have colours in $A$. This is because when we ``re-order'' a path so that  adjacent edges have ether consecutive or increasing colours, we never change the colours of the edges of the original path. Therefore, the next result follows from Lemma~\ref{orderedpaths}.

\begin{lemma}\label{reallyorderedpaths}
Let $\m$ be a maniplex and let $u,v \in F$, where $F$ is a connected component of $\M_{\overline{\ i_1,  \dots, i_k\}}}$, for some $-1=i_0 < i_1<i_2<\dots< i_k\leq n-1$. Set $i_0:=-1$ and $i_{k+1}:=n$. Then there exists $u=w_0, w_1, \dots, w_{k+1}=v \in F$ satisfying that for each $j=0,\dots k$, there is a path from $w_j$ to $w_{j+1}$ having edges of colours greater than $i_j$ but smaller than $i_{j+1}$.
\end{lemma}

We now use the strong flag connectivity of a polytope $\po$ to show a property about intersections of coloured paths of a flag graph of a polytope. 

\begin{lemma}\label{POL->WPIP}
Let $\po$ be a polytope and $\G_\po$ be its flag graph. Given $v_\Phi$ and $v_\Psi$  two vertices of $\G_\po$, if there are two paths from $v_\Phi$ to $v_\Psi$, one using edges of colours  in a set $A$ and the other one using edges of colours in a set $B$, then there is a path from $v_\Phi$ to $v_\Psi$ using colours only from $A \cap B$.
\end{lemma}

\begin{proof}
Each of the two paths corresponds to a sequence of adjacent flags starting in $\Phi$ and finishing in $\Psi$.
One path uses colours only in $A$, which means that the adjacencies of its corresponding sequence are taken all in $ A$. Hence, $\Phi_a=\Psi_a$ for every $a\in \overline{A}$. 
Similarly, $\Phi_b=\Psi_b$ for every $b\in \overline{B}$.
This means that  $\Phi_i=\Psi_i$ for all $i \in \overline{A} \cup \overline{B}=\overline{A\cap B}$, so $\Phi_i \in \Phi\cap\Psi$, for every $\overline{A\cap B}$.

As $\po$ is strongly flag connected, there is a sequence of 
adjacent flags from $\Phi$ to $\Psi$ all having the elements of $\Phi\cap\Psi$.
Hence, there is a path from $v_\Phi$ to $v_\Psi$ with edges using colours only in $A\cap B$.
\end{proof}
\begin{figure}[htbp]
\begin{center}
\includegraphics[width=7cm]{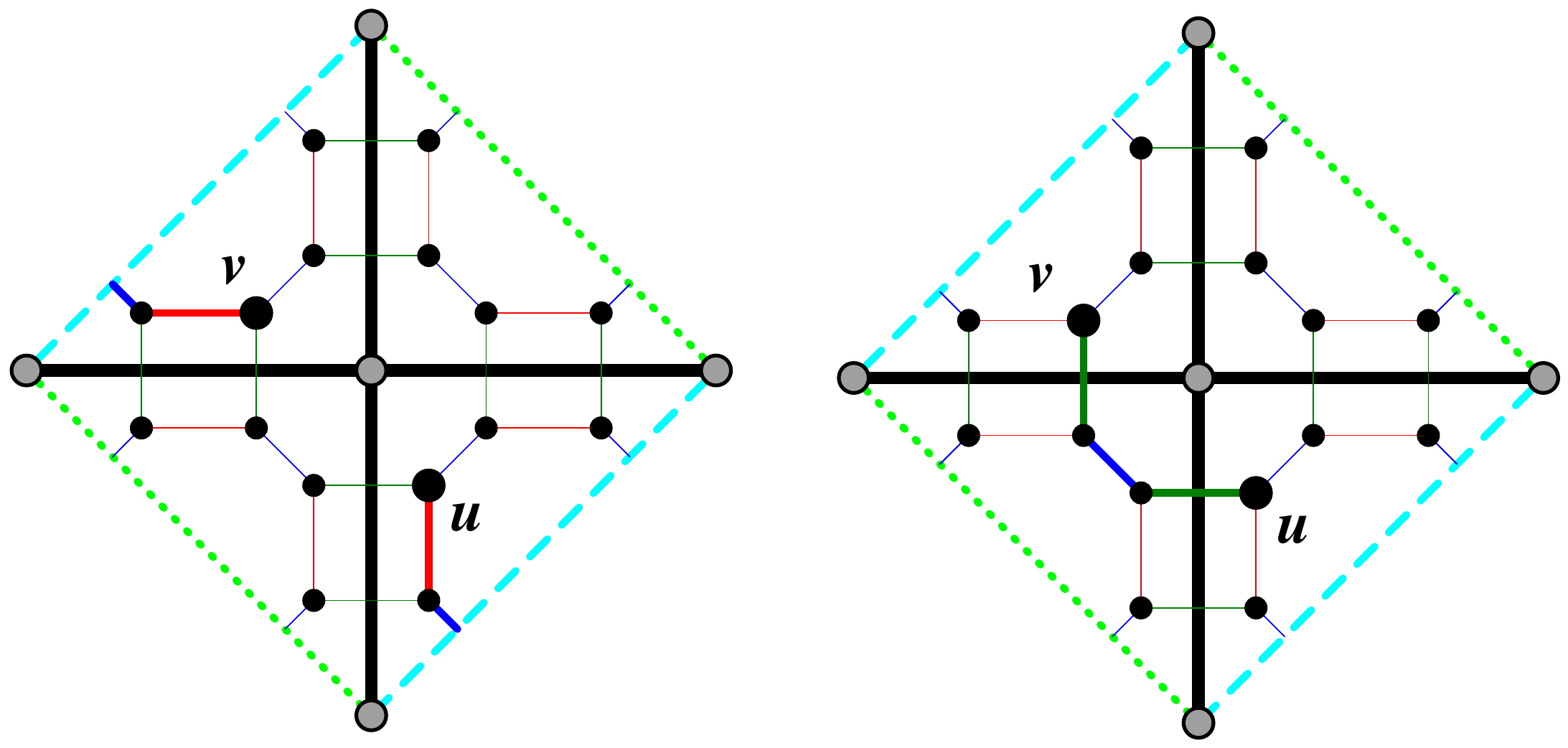}
\caption{The vertices $u$ and $v$ are joined by a path of colours $0$ and $1$ and a path of colours $1$ and $2$, but there is no path of colour $1$ between them.}
\label{PIP{4,4}(1,1)}
\end{center}
\end{figure}

This motivates us to give the following definition.

\begin{definition} 
Let $\m$ be a $n$-maniplex. We say that $\m$ has the {\em strong path intersection property (or SPIP)} if whenever there are flags $v,u\in\m$ and sets $A,B\subset\n$ such that there are two $v-u$-pahts in $\m$, one with edges of colours in $A$ and the other one with colours in $B$, then there is a $v-u$-path in $\m$ with edges of colours in $A\cap B$.
\end{definition}

Similarly to the CIP, not every maniplex has the SPIP. Examples of $3$-maniplexes that fail to have the SPIP are the maps on the torus $\{4,4\}_{(1,a)}$, with $a=0,1$ (see Figures~\ref{PIP{4,4}(1,1)} and~\ref{PIP{4,4}(1,0)}).

\begin{figure}[htbp]
\begin{center}
\includegraphics[width=9cm]{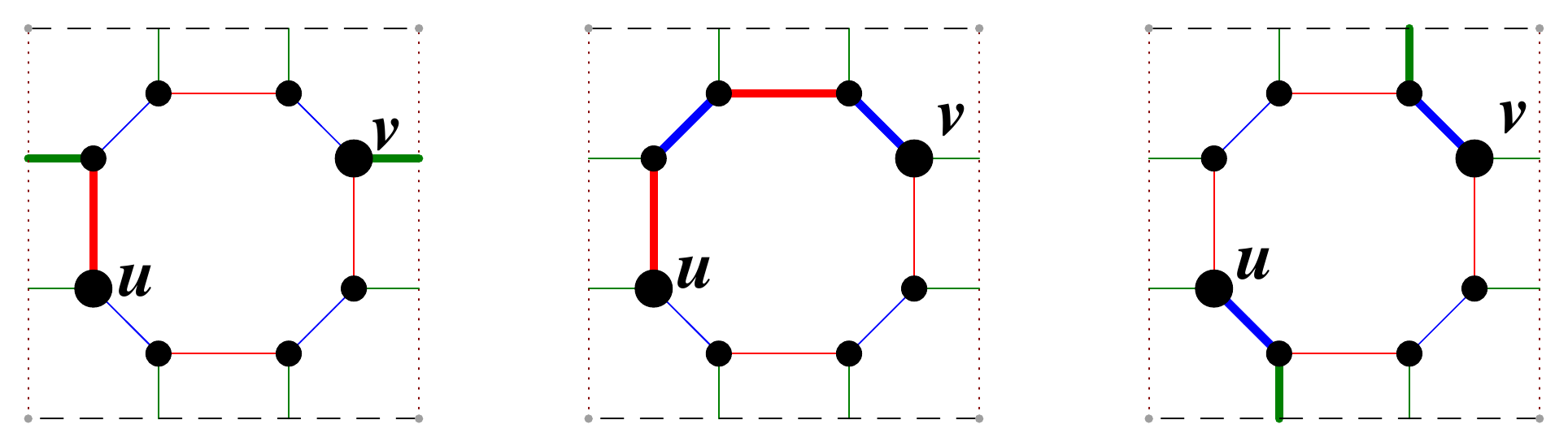}
\caption{The vertices $u$ and $v$ are joined by a path of colours $i$ and $j$, for $i\neq j$, $i,j\in\{0,1,2\}$, but no path of a single colour between them.}
\label{PIP{4,4}(1,0)}
\end{center}
\end{figure}

An example of a $4$-maniplex that fails to have the SPIP is the tessellation of the $3$-torus described in the previous section. Let $\Phi$ be the flag that, in Figure~\ref{PIP4}, consists of vertex $A$, the edge $AB$, the shadowed square on the left picture of the figure and the octahedron on the upper right corner. Then let $\Psi$ be the flag that also consists of vertex $A$, but now the edge $AF$, the shadowed square on the right picture of the figure and also the octahedron on the upper right corner. 
It is not difficult to see that one can go from $\Phi$ to $\Psi$ with a path of colours $0$, $1$ and $2$, or with a path of colours $1$, $2$ and $3$. However, a path of colours $1$ and $2$ starting at $\Phi$ must have finish at a flag whose edge is one of the edges $AB$, $AC$, $AD$ or $AE$.

\begin{figure}[htbp]
\begin{center}
\includegraphics[width=5cm]{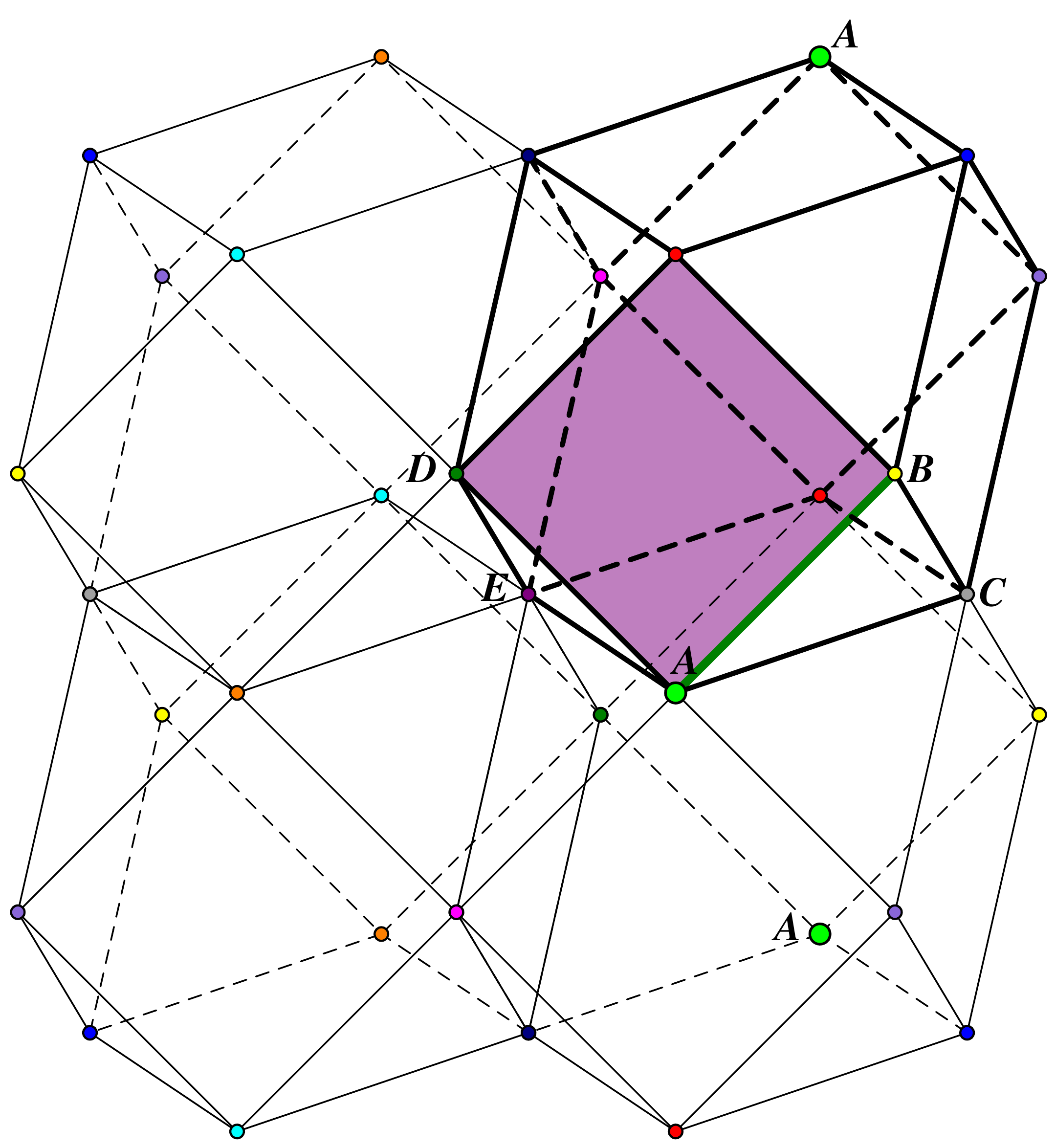} \ \ \ \ 
\includegraphics[width=5cm]{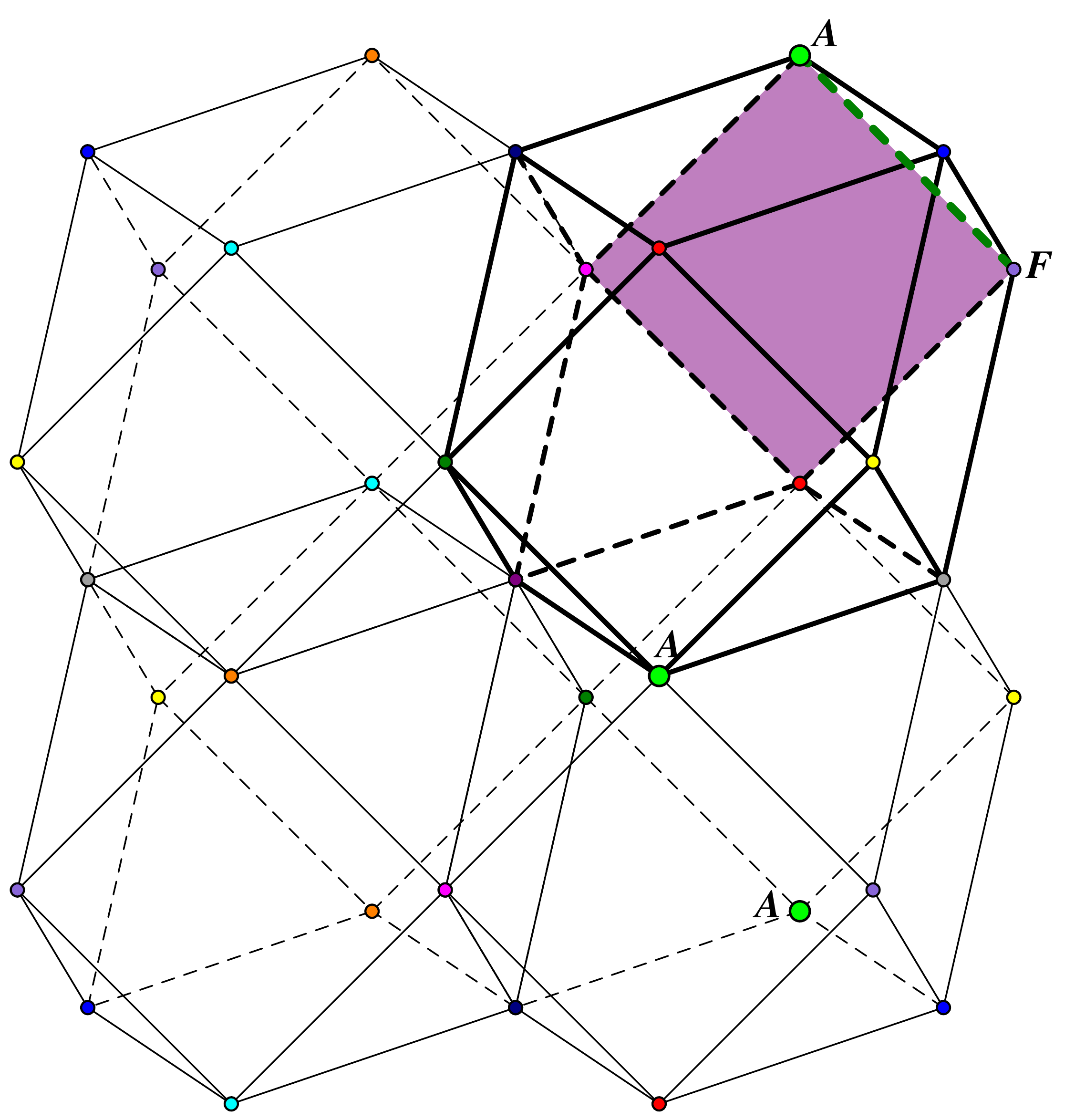}
\caption{A tessellation of the $3$-tours with octahedra and cuboctahedra. Its induced maniplex does not satisfy the SPIP.}
\label{PIP4}
\end{center}
\end{figure}

Note that whenever a maniplex $\m$ satisfies the SPIP, it satisfies a weaker version of it, namely, when $A=\{0,1,\dots, j-1\}$ and $B=\{i+1, \dots, n-1\}$, for some $0\leq i < j \leq n-1$. This motivate us to give the following definition.

\begin{definition}
A maniplex $\m$ has the {\em weak path intersection property (WPIP)} if for every two different flags $u,v$ of $\mathcal{M}$ and every $i,j \in [n]$ with $i<j$, it is satisfied that every time that there is a path from $u$ to $v$ that uses colours only in $[n]_{>i} := \{k\in [n] : k > i\}$ and another one that uses colours only in $[n]_{<j} := \{ k \in [n] : k < j\}$, then there exists a path from $u$ to $v$ that uses colours only in $[n]_{>i}\cap [n]_{<j}$.
\end{definition}

Again, the maps on the torus $\{4,4\}_{(1,a)}$, with $a=0,1$ fail to have the WPIP. 
On the other hand, every polytopal maniplex has the WPIP (as it has the SPIP). 
Thus, if a maniplex has the CIP, then it has the WPIP. 
In the following theorem we show that these two conditions are in fact equivalent. 
As a corollary of it, we have that a maniplex has the SPIP if and only if it has the WPIP.

\begin{theorem}
A maniplex $\m$ has the CIP if and only if it has the WPIP.
\end{theorem}

\begin{proof}
If $\m$ has the CIP, by Theorem~\ref{CIP<->Poly} and Lemma~\ref{POL->WPIP}, it has the SPIP, and hence is the WPIP.
The interesting part is the converse.

Let $\m$ be a maniplex that has the WPIP,
and let $\{G_1, G_2, \dots, G_k\}$ be a chain of the poset $\po_\M$. 
We need to show that $\bigcap_{j=1}^k G_{j}$ is connected. 
We proceed by induction over $k$.

For $k=1$ there is nothing to show, so let $k=2$. 
Then $G_j$ is a connected component of $\m_{\overline{i_j}}$, for some $i_1, i_2\in\n$.
Let $u,v \in G_1\cap G_2$. For each $j=1,2$, since $u,v\in G_j$, then there is a $u-v$ path with colours different from $i_j$.
By Lemma~\ref{orderedpaths} there exists $w_j$ such that there is a path $p_j$ from $u$ to $w_j$ having colours smaller than $i_j$ and a path $q_j$ from $w_j$ to $v$ having colours greater than $i_j$.
Without loss of generality assume that $i_1<i_2$.
Then, $p_1$ followed by $p_2$ is a path from $w_1$ to $w_2$ that goes through $u$ and has colours in $[n]_{<{i_2}}$. And $q_1$ followed by $q_2$ is a path from $w_1$ to $w_2$ that goes through $v$ and has colours in $\n_{>i_1}$.
Since $\m$ has the WPIP, there is a $w_1-w_2$ path $t$ with colours in $\{i_1+1, \dots, i_2-1\}$.
Note that $p_1$, $t$ and $q_2$ do not have edges of colours $i_1$ nor $i_2$.
Thus, the path that starts in $u$, follows first $p_1$, then $t$ and finally $q_2$ is a $u-v$ path that is contained in $G_1\cap G_2$. Therefore $G_1\cap G_2$ is connected.

Suppose now that $k\geq 3$ and that the intersection of the elements of any chain with less than $k$ colours is connected.
Let $u,v \in \bigcap_{j=1}^k G_{j}$. The idea is similar to the one for $k=2$.
Without loss of generality we assume that $i_1<i_2<\dots <i_k$.
As $u,v\in G_1$, there is a $u-v$ path with colours different than $i_1$.
By Lemma~\ref{orderedpaths}, there exists a vertex $w_1$ such that the $u-w_1$ path $p_1$ has colours smaller than $i_1$ and the $w_1-v$ path $p_2$ has colours greater than $i_1$.

Now, $u,v \in \bigcap_{j=2}^k G_{j}$, which, by induction hypothesis is connected. 
Hence, there is a $u-v$ path with colours in $\overline{\{i_2, i_3, \dots, i_k\}}$.
By Lemma~\ref{reallyorderedpaths}, there exist $w_2, w_3, \dots w_k$ such that there is a path $q_1$ from $u$ to $w_2$ having edges of colours smaller than $i_2$, for each $j=2,\dots, k-1$, there is a path $q_j$ from $w_j$ to $w_{j+1}$ having edges of colours greater than $i_j$ but smaller than $i_{j+1}$ and there is a path $q_k$ from $w_k$ to $v$ having colours all greater than $i_k$.

Note that $p_1$ followed by $q_1$ is a $w_1-w_2$ in which all the edges have colours smaller than $i_2$. On the other hand $p_2$ followed by $q_k, q_{k-1}, \dots, q_2$ is a $w_1-w_2$ path whose edges have colours greater than $i_1$.
As $\m$ has the WPIP, there is a $w_1-w_2$ path $t$ with colours greater then $i_1$ but smaller than $i_2$.
Similar as before, the paths $p_1$, $t$, $q_2, q_3, \dots, q_k$ do not have edges of colours $i_1, \dots, i_k$. This implies that there is a $u-v$ path with colours in $\overline{\{i_1,\dots, i_k\}}$ and hence $\bigcap_{j=1}^k G_{j}$ is connected. 
Thus, the theorem follows.
\end{proof}

\begin{coro}
\label{PIP}
A maniplex has the SPIP if and only if it has the WPIP.
\end{coro}

\begin{proof}
Clearly the WPIP is a consequence of the SPIP, so  a maniplex has the SPIP, has the WPIP.
On the other hand, if a maniplex $\m$ has the WPIP, then it has the CIP. 
Hence, $\m$ is polytopal and therefore it has the SPIP.
\end{proof}

The above corollary motivates us to say that a maniplex has the {\em Path Intersection Property (PIP)} if it has either the SPIP or the WPIP (and therefore both). 
The obvious consequence of Corollary~\ref{PIP} is that whenever we need to show that a maniplex has the PIP, it is enough to show the weak version of it, whereas if the PIP is given, we can use it in its stronger version.
The following result is a straightforward consequence of the results of this section.

\begin{theorem}
A maniplex is polytopal if and only if it has the Path Intersection Property. 
\end{theorem}

\section{On the mix or parallel product of maniplexes}

There are several operations on polytopes, that result on maniplexes that are not always polytopal. One of the most interesting ones is the mix or parallel product of maniplexes. 
The mix of maniplexes generalizes the mix of polytopes (see \cite{ARP,mixing}) and the parallel product of maps (see \cite{parallel}).
While this operation is often defined in terms of groups, it can be given in terms of the flag graphs of the polytopes in the following way.

Let $\m$ and $\cal N$ two $n$-maniplexes. 
Consider the product of the colour graphs $\m$ and $\cal N$ whose vertices are the vertices of $V(\m) \times V({\cal N})$ and where the vertices $(a,u), (b,v) \in V(\m) \times V({\cal N})$ are join by an $i$-edge whenever $a,b\in V(\m)$ and $u,v\in V({\cal N})$ are join by an $i$-edge.
In other words, for each $i\in [n]$,  $(b,v)^i=(a,u)$ if and only if $b^i=a$ and $v^i=u$.

The {\em mix or parallel product} $\m\diamondsuit {\cal N} $ of $\m$ and $\cal N$ is a connected component of the product described above.

It is not difficult to see that if a maniplex $\m$ is a cover of another maniplex $\cal N$ (as edge-coloured graphs), then the mix $\m\diamondsuit {\cal N}$ gives us again $\m$. 
In particular this implies that if $\m$ is a polytopal maniplex that covers a non-polytopal maniplex $\cal N$, then their mix is polytopal. 
In other words, when mixing maniplexes is not necessary to have both maniplexes polytopal in order to have the mix to be polytopal. 

\begin{figure}[htbp]
\begin{center}
\includegraphics[width=14cm]{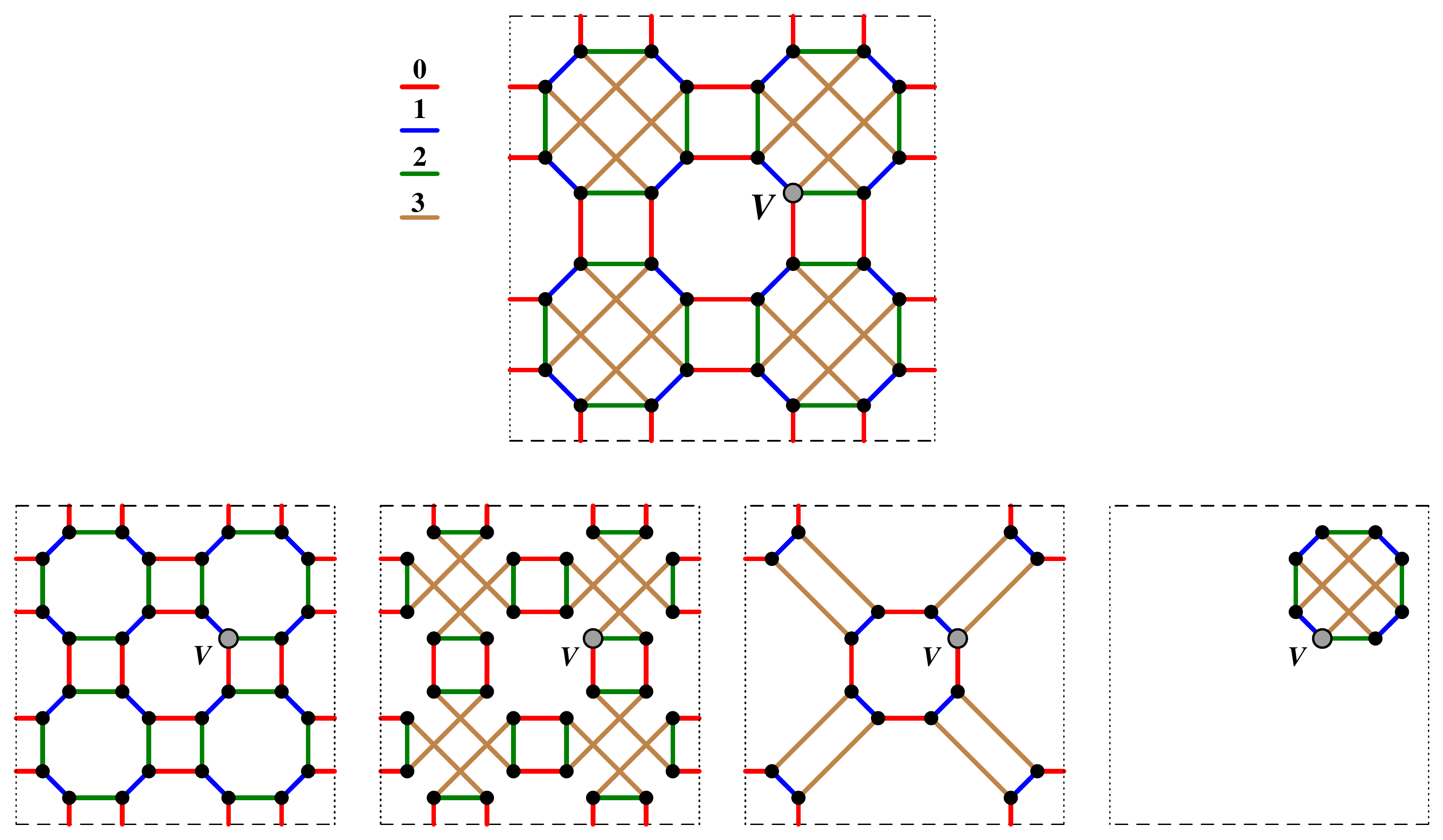}
\caption{The non-polytopal maniplex $\m$ and $i$-faces of $\m$, for each $i=0,1,2,3$.}
\label{ManiM}
\end{center}
\end{figure}

Moreover, the mix $\po=\m\diamondsuit {\cal N}$ of the non-polytopal maniplexes $\m$ and $\cal N$ given in Figures~\ref{ManiM} and~\ref{ManiN}, respectively, is polytopal. In figure~\ref{PolyMix} we present a part of $\po$, and emphasize the $i$-faces containing a given {\em base} flag. From the figures one can verify that when taking connected components that contain the base flag satisfy the CIP. Using the symmetries of $\m$ and $\cal N$, we then obtain the polytopality of $\po$. (The graphs of Figures~\ref{ManiM},~\ref{ManiN} and ~\ref{PolyMix} are drawn on a torus, so the dotted lines of the figures are identified.)

\begin{figure}[htbp]
\begin{center}
\includegraphics[width=6cm]{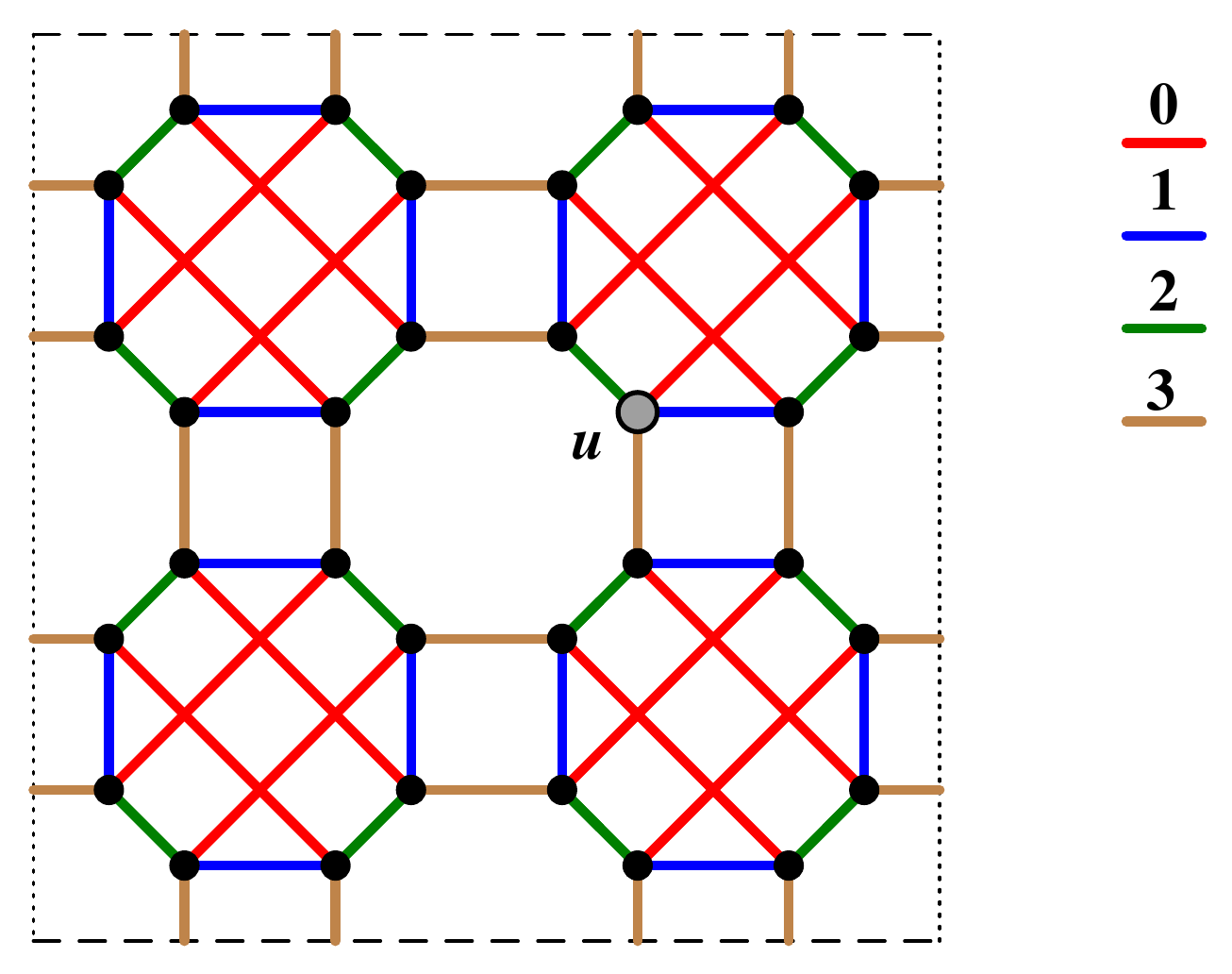}
\caption{The non-polytopal maniplex $\cal N$.}
\label{ManiN}
\end{center}
\end{figure}

\begin{figure}[htbp]
\begin{center}
\includegraphics[width=6cm]{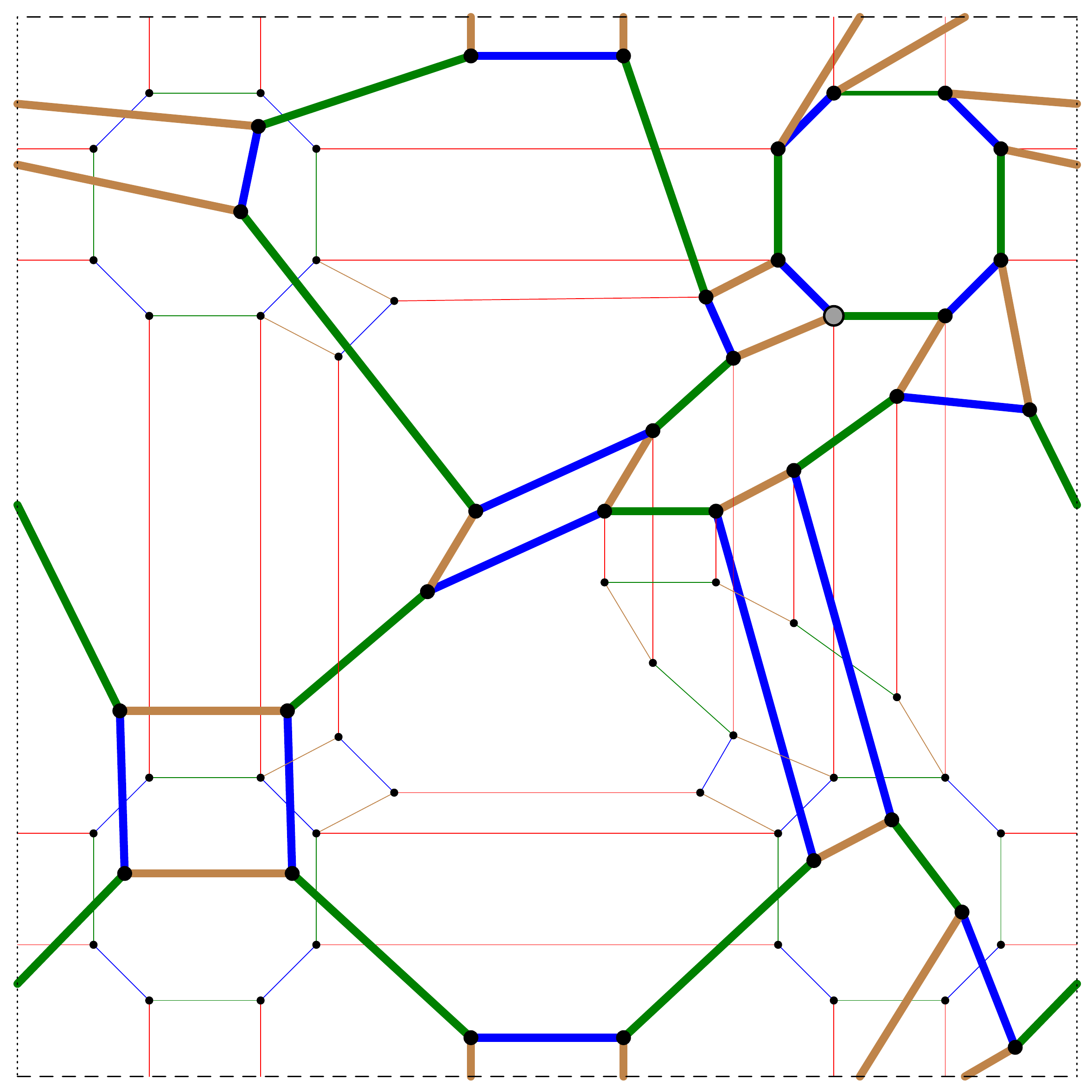} \ \ \ \ \ \ \ \ 
\includegraphics[width=6cm]{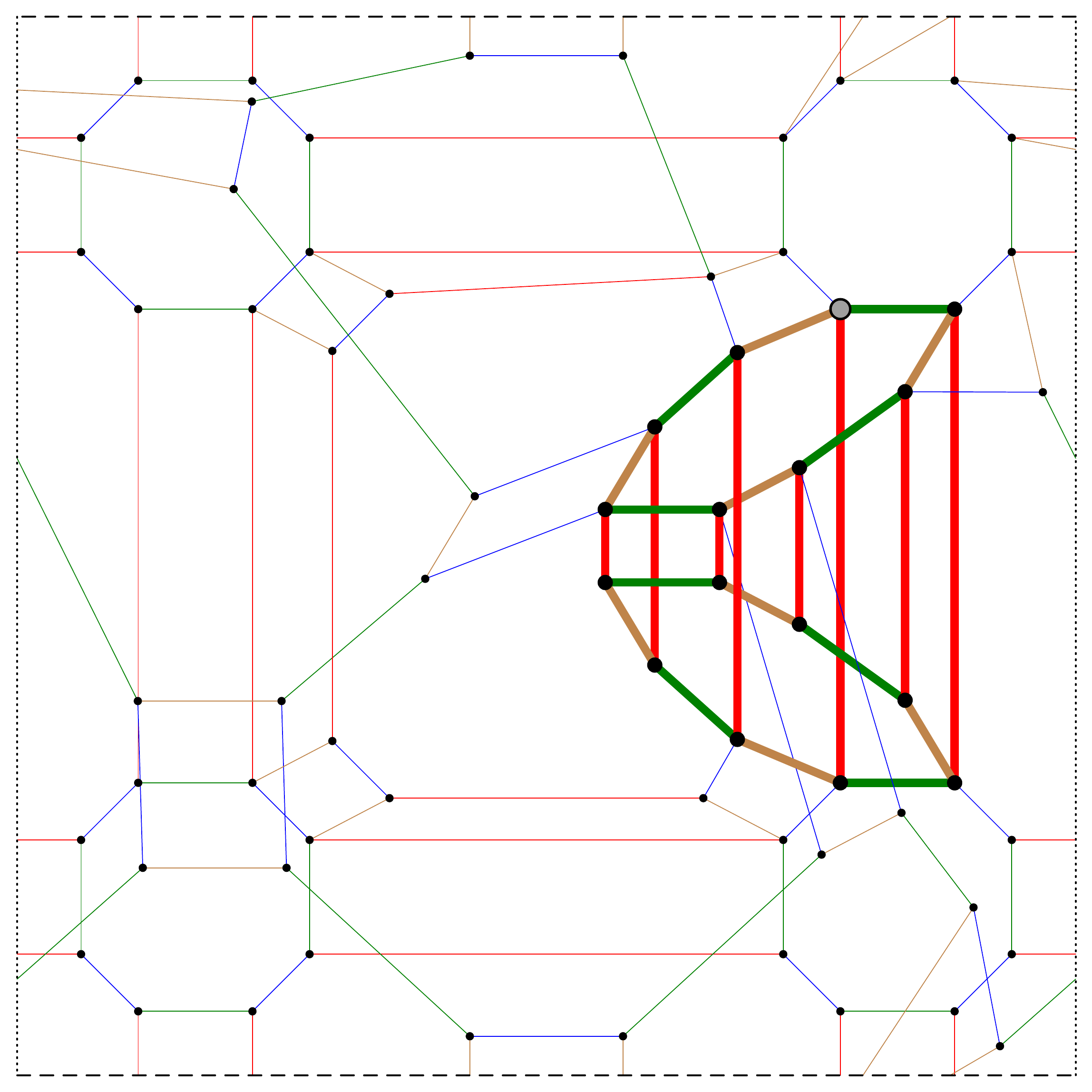} 
\includegraphics[width=6cm]{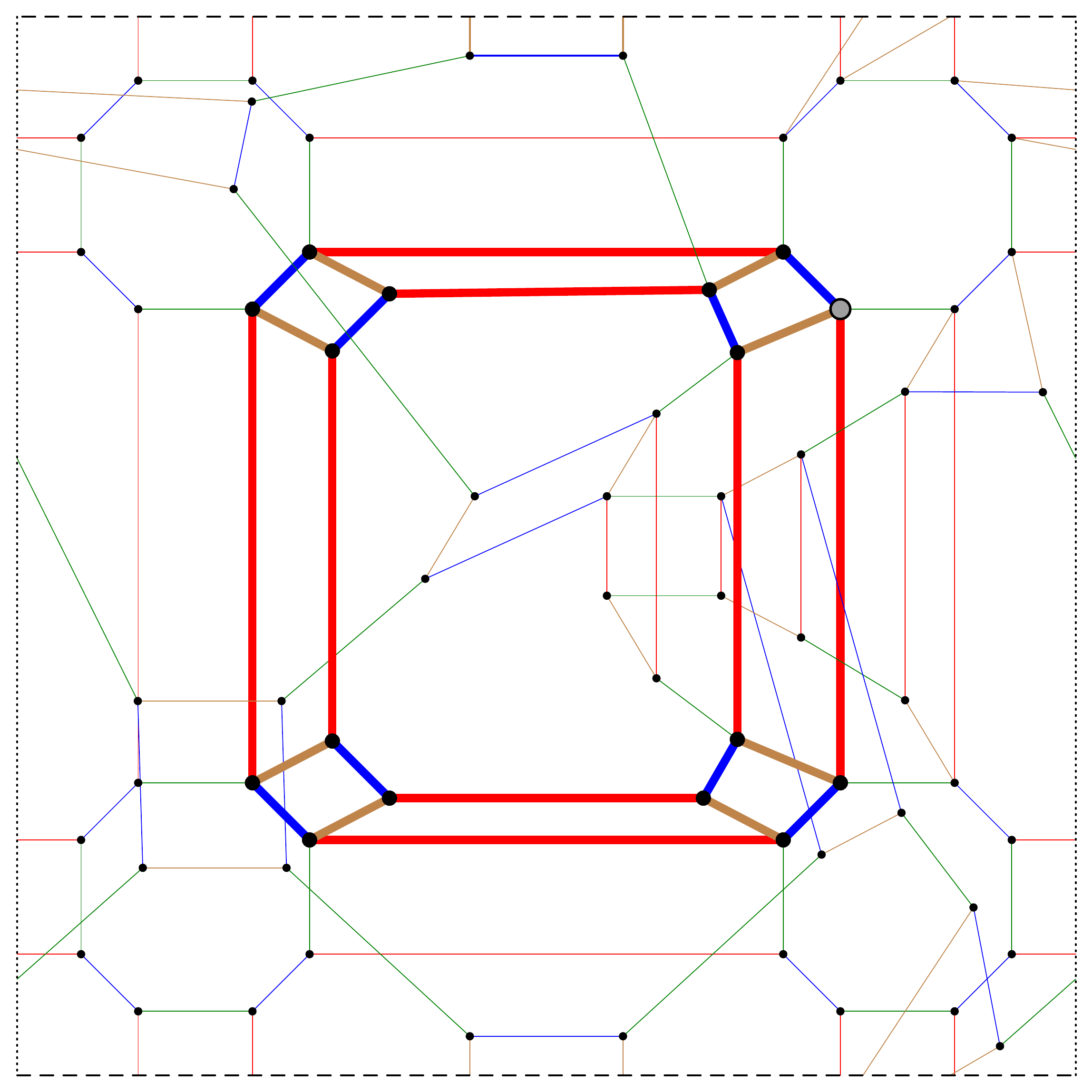} \ \ \ \ \ \ \ \ 
\includegraphics[width=6cm]{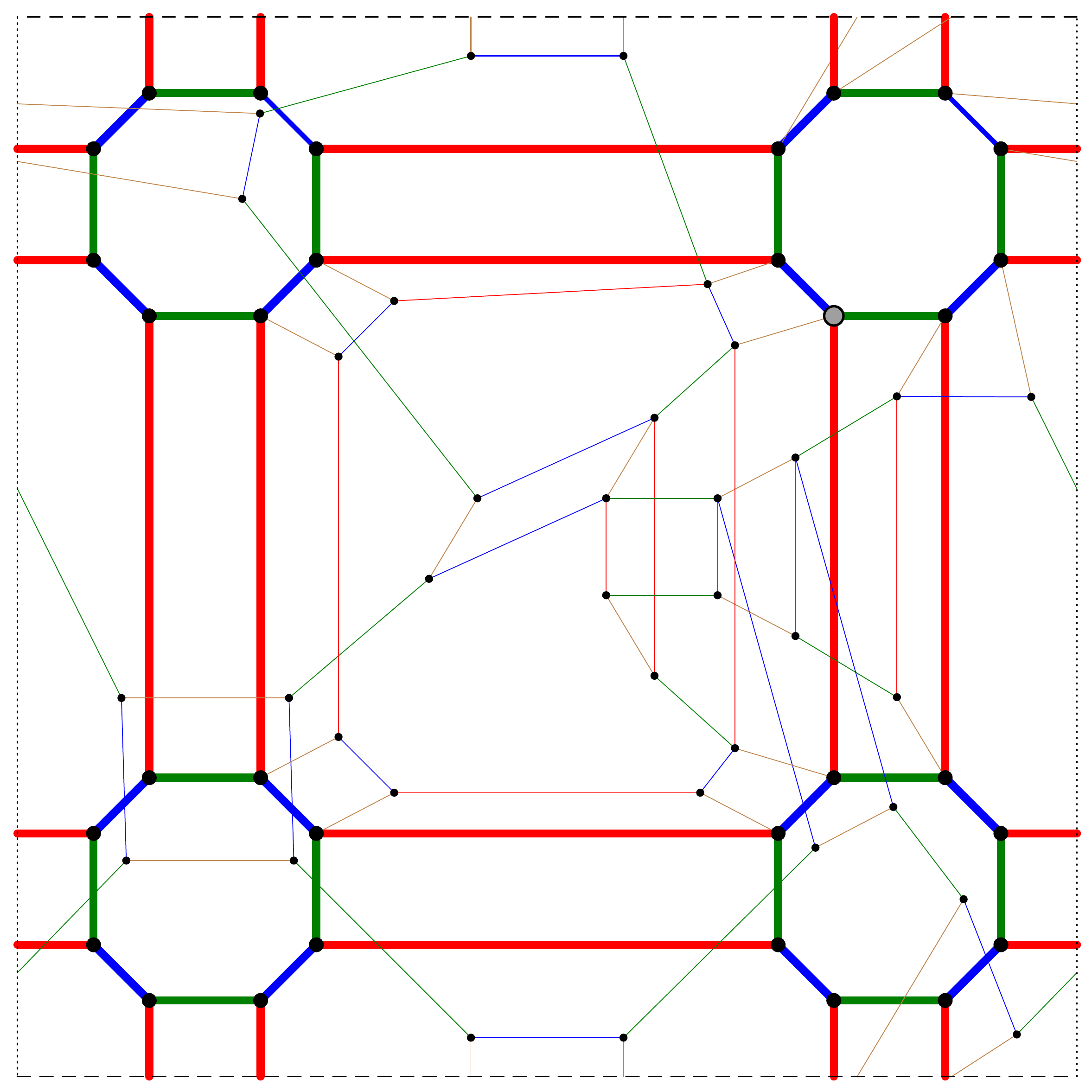}
\caption{For each $i=0,1,2,3$, an $i$-face of the polytopal maniplex $\po=\m\diamondsuit {\cal N}$. }
\label{PolyMix}
\end{center}
\end{figure}

In \cite{mixing}, the authors study instances of when the mix of two polytopes is again a polytope, and give a small example of two (regular) $4$-polytopes whose mix is not polytopal.  Their study is not in terms of maniplexes, but in terms of posets that satisfy the properties of an abstract polytope except for the strong flag connectivity. The general question of determining when a mix of two polytopes is polytopal remains open.

\begin{problem} Give necessary and sufficient conditions on two maniplexes (as coloured graphs) in order to have a polytopal mix.
\end{problem}

\section*{Acknowledgments}
We gratefully acknowledge financial support of the PAPIIT-DGAPA, under grant IN107015, and of CONACyT, under grant 166951.
The first author wishes to thank particularly Ian Gleason for his valuable inputs while this work was being developed, and Octavio Arizmendi and Ian Gleason for their important suggestions on the first draft of this paper.  
The completion of this work was done while the second author was on sabbatical at the Laboratoire d'Informatique de l'\'Ecole Polytechnique. She thanks LIX and Vicent Pilaud for their hospitality, as well as the program PASPA-DGAPA and the UNAM for the support for this sabbatical stay.

\end{document}